\documentclass[12pt,a4paper]{amsart}

\usepackage{amsmath,amsthm,amsfonts,amssymb,amscd,euscript,enumerate}
\usepackage{mathtools}
\usepackage{cite}
\usepackage{pgfplots}
\usetikzlibrary{patterns,arrows,positioning,calc,fadings,shapes,decorations.markings}
\usepackage{tikz-cd}
\usepackage{tikz}
\usepackage{graphicx}
\usepackage{csquotes}

\theoremstyle{plain}
\newtheorem{thm}{Theorem}[section]
\newtheorem{theorem}[thm]{Theorem}
\newtheorem{lemma}[thm]{Lemma}

\theoremstyle{definition}

\newtheorem{remark}[thm]{Remark}

\newtheorem{definition}[thm]{Definition}

\newtheorem{defn-thm}[thm]{Definition-Theorem}

\begin{document}
	\title{A module-theoretic interpretation of quantum expansion formula}
	\author{Yutong Yu}
	\address{Department of Mathematical Sciences\\
		Tsinghua University\\
		Beijing 100084, P.~R.~China}
	\email{yyt20@mails.tsinghua.edu.cn(Y. Yu)}
\begin{abstract}
	We provide a module-theoretic interpretation of the expansion formula given by Huang (2022), which defines a map on perfect matchings to compute the expansion of quantum cluster variables in quantum cluster algebras arising from unpunctured surfaces. In addition, we present a multiplication formula for string modules with one-dimensional extension space, derived using the skein relations. For the Kronecker type, an alternative expansion formula was given in Çanakçı and Lampe (2020), and we show that the two expansion formulas coincide.
\end{abstract}

\maketitle

\textbf{Keywords:} Quantum cluster algebras; string modules; perfect matchings; canonical submodules; multiplication formulas.

\textbf{MSC (2020):} 13F60; 16G20.

\section{INTRODUCTION}
\par Cluster algebras were introduced by Fomin and Zelevinsky around 2000 in \cite{Fomin2001ClusterAI}. Later, Berenstein and Zelevinsky introduced quantum cluster algebras in \cite{BERENSTEIN2005405}. The theory of cluster algebras is related to many other branches of mathematics, such as Lie theory, representation theory of algebras, Teichmüller theory and mathematical physics.
\par Cluster algebras coming from marked surfaces were studied in \cite{10.1007/s11511-008-0030-7}. In this paper, we study surfaces without punctures. Consider a compact connected oriented 2-dimensional bordered Riemann surface $S$ and a finite set of marked points $M$ lying on the boundary of $S$ with at least one marked point on each boundary component. In \cite{10.1007/s11511-008-0030-7}, a cluster algebra $\mathcal{A}(S,M)$ corresponds to the marked surface $(S,M)$ with a triangulation $\Gamma$. The main result is that there is a bijection $\gamma \mapsto x_\gamma$ between the internal arcs in $S$ and the cluster variables in $\mathcal{A}(S,M)$, furthermore triangulations are in bijection with clusters and flips correspond to mutations. Later, \cite{Schiffler2007OnCA} and \cite{SCHIFFLER20101885} invented the notion of a complete path and gave an expression of $x_\gamma$ for the cluster algebra $\mathcal{A}(S,M)$ with principal coefficients. Later \cite{Musiker2008ClusterEF} gave an expansion formula using perfect matchings which is essentially the same. The expansion formula was then generalized to marked surfaces in \cite{MUSIKER20112241}. The quantum cluster algebras arising from a marked surface are natural, and all the results have their counterparts in the quantum case. Then Huang gave the expansion formula for quantum cluster variables in \cite{Huang2021AnEF}.
\par The Jacobian algebras associated with the quivers with potential given by the triangulations of surfaces are gentle algebras. The module categories of gentle algebras are studied in \cite{10.2140/ant.2010.4.201} and the cluster category $C$ of a quiver with potential is studied in \cite{Amiot2008ClusterCF}. A correspondence between (closed) curves on surfaces and objects in cluster categories was given by \cite{zbMATH06000314}. Specifically, the string objects in $C$ are indexed by curves and can be identified with string modules in the gentle algebra associated with a triangulation. Palu \cite{Palu2008} defines a cluster character for any object in some 2-Calabi-Yau category. Later, Qin \cite{Qin_2012} gives quantum Laurent polynomials for all cluster variables using Serre polynomials.
\par Then the connection between expansion formulas for arcs and cluster characters of modules should be studied as they both provide Laurent expansion formulas of cluster variables in the same cluster algebra. This question is addressed through a comparison between perfect matchings and canonical submodules in \cite{CANAKCI2021102094}. In the quantum case, Huang introduced an expansion formula for arcs, which is defined recursively on the lattice of perfect matchings of a snake graph. The connection between perfect matchings and canonical submodules still holds. The problem is how to interpret both the recursions themselves and the values arising from them in terms of canonical submodules.
\par The aim of this paper is to provide a module-theoretic interpretation of Huang's expansion formula. The key point is to find a way to define the function $v_\gamma$ on canonical submodules. Huang defined a function $v$ on perfect matchings by an initial condition and a recursive formula. His formula depends on counting the number of edges with some fixed label. The correspondence between perfect matchings and canonical submodules in \cite{CANAKCI2021102094} could give a natural way to define $v_\gamma$ by counting the vertices of canonical submodules with some fixed label. Specifically, let $\gamma$ be an arc in $(S,M)$ with string $w$. Let $M(w)$ be the corresponding string module and $CS(M(w))$ be the set of canonical submodules.
\begin{theorem}
	There exists a map $v_\gamma:CS(M(w))\rightarrow \mathbb{Z}$ such that
	\begin{enumerate}
		\item $v_\gamma(0)=0$.
		\item If two submodules $N$ and $N'$ satisfy $N=N'\cup {x_j}$ for some index $j$, then
		\begin{equation*}
			v_\gamma(N)-v_\gamma(N')=\Omega'(x_j,N)
		\end{equation*}
	\end{enumerate}
\end{theorem}
We show that the map $v_\gamma$ coincides with the map $v$ in \cite{Huang2021AnEF}. As a consequence, we obtain the following result.
\begin{theorem}
	The quantum cluster variable associated with $M(w)$ can be expressed by
	\begin{equation*}
		X_{M(w)}=\sum_{N\in CS(M(w))}q^{\frac{v_\gamma(N)}{2}}X^\Gamma(N)
	\end{equation*}
\end{theorem}
\par Cluster multiplication formulas play an important role in cluster theory and have attracted significant attention. 
In the context of acyclic cluster algebras, Sherman and Zelevinsky~\cite{zbMATH02150997} first established such formulas for rank~2 cluster algebras of finite and affine types, which were later extended to the affine type $A_2^{(1)}$ by Cerulli~\cite{zbMATH06119057}. 
Caldero and Keller~\cite{CalderoKeller2008} then constructed cluster multiplication formulas for finite types. 
This result was subsequently generalized to acyclic types by Xiao and Xu~\cite{XiaoXu2010}, as well as Xu~\cite{zbMATH05678490}. 
In the framework of acyclic quantum cluster algebras, Ding and Xu~\cite{zbMATH06102811} first obtained cluster multiplication formulas for the Kronecker case. 
Later, Bai, Chen, Ding, and Xu~\cite{zbMATH07074193} extended this result to the affine type $A_2^{(2)}$.

\par The relationship between skein algebras and quantum cluster algebras was established in \cite{Muller2016}. In this framework, the multiplication formulas admit a natural interpretation via the bijection between arcs and cluster variables. We explicitly describe the multiplication formulas by identifying the strings of the modules appearing in the resulting expressions.
\par For the Kronecker case, \cite{CANAKCI2020105132} provided an expansion formula for quantum cluster variables in the corresponding quantum cluster algebra using a different approach. They introduce a direct method to compute the map $v$ via symmetric difference. Although the two maps are not identical, it can be shown that they yield the same expansion formula.
\section{REVIEW OF CLUSTER ALGEBRA}
We will work over a fixed algebraically closed field $K$. Let $q$ be a formal variable, $\widetilde{B}=(b_{ij})$ be a $m\times n$ matrix with full rank, and $\Lambda$ be a skew-symmetric matrix satisfying
\begin{equation}
	\Lambda \widetilde{B}=-\begin{bmatrix}
		D \\0
	\end{bmatrix} \nonumber
\end{equation}
where $D$ is a diagonal matrix with positive entries. We call $(\widetilde{B},\Lambda)$ a compatible pair. The submatrix $B=(b_{ij})$ with ${1\leq i\leq n,1\leq j\leq n}$ is called the principal part of $\widetilde{B}$.
\begin{definition}
	Let $L$ be a lattice of rank $m$ with a basis $\left\{e_i|1\leq i\leq m \right\}$ with a skew-symmetric bilinear form $\Lambda$. The $quantum \ torus$ $\mathcal{T}=\mathcal{T}(L,\Lambda)$ is the $\mathbb{Z}[q^{\pm \frac{1}{2}}]$-algebra generated by $X^g$, $g\in L$ subject to the relation
	\begin{equation*}
		X^gX^h=q^{\Lambda(g,h)/2}X^{g+h}
	\end{equation*}
\end{definition}
\begin{remark}
	The quantum torus admits an involutive $\mathbb{Z}[q^{\pm \frac{1}{2}}]$-algebra automorphism, called the $\mathbf{bar \ involution}$, induced by $q\mapsto q^{-1}$ and $X^a\mapsto X^a$ for all $a\in L$. We say an element $f$ is $\mathbf{bar}$-$\mathbf{invariant}$ if $f$ is invariant under bar involution.
\end{remark}
\par An initial $quantum \ seed$ is a triple $(\Lambda,\widetilde{B},X)$ such that the pair $(\widetilde{B},\Lambda)$ is compatible and $X=\left\{X_1,...,X_m\right\}$ where $X_i=X^{e_i}$ for $1\leq i\leq m$.
\par For any $1\leq k\leq n$, we can define the $mutation$ of the quantum seed in direction $k$. Denoted by $\mu _k(\Lambda,\widetilde{B},X)=(\Lambda^{\prime},\widetilde{B}^{\prime},X^{\prime})$, where
\begin{enumerate}
	\item \begin{equation*}
		\Lambda'_{ij}=\left\{\begin{aligned}
			&\Lambda_{ij},      &\mbox{if $i,j\neq k$}\\
			&\Lambda(e_i,-e_k+\sum_{l}{[b_{lk}]_+}e_l), &i\neq k=j \\
		\end{aligned}
		\right.
	\end{equation*}
	\item \begin{equation*}
		b^{\prime}_{ij}=\begin{cases}
			-b_{ij} &\mbox{if $i=k$ or $j=k$}\\
			b_{ij}+[b_{ik}]_+b_{kj}+b_{ik}[-b_{kj}]_+ &otherwise
		\end{cases}
	\end{equation*}
	where $\widetilde{B}=(b_{ij})$ and $\widetilde{B}^{\prime}=(b^{\prime}_{ij})$.
	\item $X'=\left\{X'_1,...,X'_m\right\}$ is given by
	\begin{equation*}
		X'_k=X^{-e_k+\sum_{i}{\left[b_{ik}\right]_+e_i}}+X^{-e_k+\sum_{i}{\left[-b_{ik}\right]_+e_i}}
	\end{equation*}
	where $X'_i=X_i$ for $i\neq k$.
\end{enumerate}
It is proved in \cite{BERENSTEIN2005405} that the mutation here can be applied on the new seed on any direction. If we collect all the seeds which can be obtained from a finite sequence of mutations on the initial seed, we can build an $n$-regular tree $\mathbb{T}_n$ where the vertices correspond to the quantum seeds and edges correspond to mutations in some direction. We denote $t_0$ the vertex of the initial seed. For any vertex $t$ in $\mathbb{T}_n$, the quantum seed $X(t)=(X_1(t),\cdots,X_m(t))$ can be obtained by a finite sequence of mutations in direction $k_1,\cdots,k_s$ and we denote by $t=\mu_{k_s}\cdots \mu_{k_1}(t_0)$.
\begin{definition}
	The quantum cluster algebra $\mathcal{A}_q$ is a subalgebra of $\mathcal{T}$ generated by $X_i(t)$ for any $1\leq i\leq m$ and $t\in \mathbb{T}_n$.
\end{definition}
By definition, all $X_i(t)$ are equal for any $n+1\leq i\leq m$ and $t\in \mathbb{T}_n$ and let $X_{n+1},\cdots,X_m$ denote the common value.
\begin{enumerate}
	\item $X_i(t)$ is called a quantum cluster variable for $1\leq i\leq n$ and $t\in \mathbb{T}_n$.
	\item $X_i$ is called a frozen variable for $n+1\leq i\leq m$.
	\item $X(t)$ is called a cluster.
	\item For any $t\in \mathbb{T}_n$ and $a\in \mathbb{N}^m$, 
	\begin{equation*}
		X(t)^a=q^\frac{-\sum_{1\leq i<j\leq m}\Lambda(a_ie_i,a_je_j)}{2}X_1(t)^{a_1}X_2(t)^{a_2}\cdots X_m(t)^{a_m}
	\end{equation*}
	is called a quantum cluster monomial.
\end{enumerate}
\begin{remark}
	When $q=1$, the above recovers the classical cluster algebra in \cite{Fomin2001ClusterAI}. We replace $X$ with $x$ to serve as the counterpart in classical cluster algebra.
\end{remark}
\begin{theorem}[\cite{BERENSTEIN2005405}]
	Any quantum cluster variable is a Laurent polynomial in $X_i(t), 1\leq i\leq m$ for some fixed $t\in \mathbb{T}_n$.
\end{theorem}

\section{GENTLE ALGEBRA FROM SURFACE}
We first recall some results from \cite{10.2140/ant.2010.4.201}. Let $S$ be an oriented surface with boundary $\partial S$, and $M$ be a non-empty finite set of points on $\partial S$ intersecting each connected component of the boundary $\partial S$. The pair $(S,M)$ is referred to as an unpunctured bordered surface with marked points. 
\par An arc in $(S,M)$ is a curve $\gamma$ in $S$ such that
\begin{itemize}
	\item the endpoints of $\gamma$ are marked points in $M$.
	\item $\gamma$ does not intersect itself, except that its endpoints may coincide.
	\item $\gamma$ intersects the boundary only in its endpoints
	\item $\gamma$ does not cut out a monogon (that is, $\gamma$ is not contractible into a point of $M$).
\end{itemize}
We call an arc $\gamma$ a $boundary \ arc$ if it cuts out a digon (that is, $\gamma$ is homotopic to a curve on the boundary that intersects $M$ only on its endpoints). Otherwise, $\gamma$ is said to be an $internal \ arc$. Each arc is considered up to homotopy in the class of such curves. A $triangulation$ of $(S,M)$ is a maximal collection $\Gamma$ of arcs that does not intersect in the interior of $S$. We only consider marked surfaces that admit a triangulation.
\begin{definition}
	$Q(\Gamma)$ is the quiver whose set of points is the set of internal arcs of $\Gamma$, and the arrows are defined as follows: whenever there is a triangle $T$ in $\Gamma$ containing two arcs $a$ and $b$, then $Q(\Gamma)$ contains an arrow $a\rightarrow b$ if $b$ is a predecessor of $a$ with respect to clockwise orientation at the joint vertex of $a$ and $b$ in $T$.
\end{definition}
A triangle $T$ in $\Gamma$ is called an $internal \ triangle$ if all edges of $T$ are internal arcs. Every internal triangle $T$ in $\Gamma$ gives rise to an oriented cycle $\alpha_T\beta_T\gamma_T$ in $Q(\Gamma)$. We define 
\begin{equation*}
	W(\Gamma)=\sum_T\alpha_T\beta_T\gamma_T
\end{equation*}
where the sum runs over all internal triangles $T$ of $\Gamma$. Then we define $A(\Gamma)$ to be the Jacobian algebra of the quiver with potential $(Q(\Gamma),W(\Gamma))$.
\begin{lemma}
	[\cite{10.2140/ant.2010.4.201}]
	The algebra $A(\Gamma)$ is a gentle algebra.
\end{lemma}
\par Recall that a finite dimensional algebra $A=KQ/I$ is $gentle$ if the following conditions are satisfied
\begin{enumerate}
	\item At each vertex of $Q$, at most two arrows start and at most two arrows stop.
	\item The ideal $I$ is generated by paths of length 2.
	\item For each arrow $\beta$ there is at most one arrow $\alpha$ and at most one arrow $\gamma$ such that $\alpha\beta\in I$ and $\beta\gamma\in I$.
	\item For each arrow $\beta$ there is at most one arrow $\alpha$ and at most one arrow $\gamma$ such that $\alpha\beta\notin I$ and $\beta\gamma\notin I$.
\end{enumerate}
\par Let $Q_0$ be the set of points in $Q$ and $Q_1$ be the set of arrows in $Q$. For an arrow $a\in Q_1$, $s(a)$ denotes the start point and $e(a)$ denotes the end point. $a^{-1}$ denotes the formal inverse of $a$ which means $s(a^{-1})=e(a)$ and $e(a^{-1})=s(a)$. The set of formal inverse arrows in $Q_1$ is denoted by $Q_1^{-1}$. A word $w=a_1a_2\cdots a_n$ is a $string$ if either $a_i$ or $a_i^{-1}$ is an arrow in $Q_1$, $a_{i+1}\neq a_i^{-1}$, $e(a_i)=s(a_{i+1})$ for $1\leq i\leq n-1$ and no subword of $w$ or its inverse is in $I$.
\par Given a string $w$, we denote by $M(w)$ the corresponding string module over $A$: the underlying vector space is obtained by replacing each vertex of $w$ by a copy of the field $K$, and the action of an arrow $a\in Q_1$ on $M(w)$ is the identity morphism if $a$ or its inverse is a letter of $w$, and zero otherwise.
\begin{remark}
	Note that by definition $M(w)=M(w^{-1})$. For each arc $\gamma$ in $S$, the corresponding string consists of the vertices where $\gamma$ crosses $\Gamma$, listed in order. The string module corresponding to a trivial string given by a vertex $i$ in $Q$ is the simple $A$-module corresponding to $i$. All the modules mentioned below are string modules.
\end{remark}
\begin{definition}
	The submodule $N$ of $M$ is a canonical embedding if the injective map $N\rightarrow M$ is induced by the identity on the non-zero component of $N$. In this way, $N$ is called a canonical submodule of $M$. Denote by $CS(M)$ the set of canonical submodules of $M$. 
\end{definition}
\par In \cite{Brstle2011AMI}, the authors provided a more direct way to define the canonical submodules. That is, for a string $w$
\begin{equation*}
	v_1-v_2-\cdots-v_s
\end{equation*}
a subword $w_I$ of $w$ indexed by an interval $I=[i,j]=\left\{i,i+1,\cdots,j\right\}$ is a string given by
\begin{equation*}
	v_i-v_{i+1}-\cdots-v_j
\end{equation*}
If in addition $M(w_I)$ is a submodule of $M(w)$, we call $w_I$ a $substring$. More generally, for any subset $I\subset \left\{1,2,\cdots,s\right\}$, we can uniquely write $I$ as a disjoint union of intervals of maximal length $I=I_1\cup I_2\cup \cdots \cup I_t$ such that
\begin{enumerate}
	\item $I_l$ is an interval for $1\leq l\leq t$.
	\item $max\left\{i|i\in I_l\right\}+2\leq min\left\{i|i\in I_{l+1}\right\}$ for each $1\leq l\leq t-1$.
	\item $I_j\cap I_k=\emptyset$ if $j\neq k$.
\end{enumerate}
\par Then for any subset $I$ with the decomposition $I=I_1\cup I_2\cup\cdots \cup I_t$, consider the string module
\begin{equation*}
	M_I(w)=\bigoplus^t_{l=1}M(w_{I_l})
\end{equation*}
Set
\begin{equation*}
	S(w)=\left\{I\subset \left\{1,2,\cdots ,s\right\}|M_I(w) \ is \ a \ submodule \ of \ M(w)\right\}
\end{equation*}
It is clear that each substring induces a canonical submodule of $M(w)$. Moreover, since the decomposition is disjoint with maximal length, the supports of the substring modules are pairwise disjoint. So any subset of $S(w)$ induces a canonical submodule of $M(w)$. The following lemma shows that canonical submodules correspond to elements of $S(w)$.
\begin{lemma}
	For any string $w$ in $Q$, there is a bijection 
	\begin{equation*}
		f:S(w)\rightarrow CS(M(w))
	\end{equation*}
\end{lemma}
\begin{proof}
	We construct the inverse of the map $f$. For any canonical submodule $N\hookrightarrow M(w)$, let
	\begin{equation*}
		I_N=\left\{i|v_i \in N\right\}
	\end{equation*}
	Then if we have $N=N(w_1)\oplus N(w_2)\oplus\cdots\oplus N(w_t)$, then each $w_i$ must be a substring of $w$. So $I_N$ has decomposition $I_1\cup I_2\cup\cdots \cup I_t$ with $w_{I_i}=w_i$ for any $1\leq i\leq t$. Therefore $f(I_N)=N$. $I_{f(I)}=I$ is obviously true for any $I\in S(w)$.
\end{proof}
From now on, we will identify a canonical submodule with its index set. Sometimes, we will use the index set to represent a string module (which need not necessarily be a submodule) whose support is the index set, provided there is no ambiguity.
\section{CLUSTER EXPANSION FORMULA}
In this section, we recall from \cite{10.1007/s11511-008-0030-7} the cluster algebra coming from an unpunctured surface $(S,M)$. Given a triangulation $\Gamma$ of a surface, let $B(\Gamma)$ denote the matrix associated with the quiver $Q(\Gamma)$. That is, $B(\Gamma)=(b_{ij})$ is an $n\times n$ matrix where:
\begin{equation*}
	b_{ij}=\mbox{number of arrows from $j$ to $i$}-\mbox{number of arrows from $i$ to $j$}
\end{equation*}
\par We say that a quantum cluster algebra $\mathcal{A}_q$ is coming from $(S,M)$ if there exists a triangulation $\Gamma$ such that $B(\Gamma)$ is the principal part of the matrix of a seed of $\mathcal{A}_q$. We let $\mathcal{A}_q(S,M)$ denote the quantum cluster algebra coming from $(S,M)$.
\par The snake graph is a significant tool in \cite{MUSIKER20112241} to prove positivity for cluster algebras from surfaces, and in \cite{Huang2021AnEF} for the quantum version. Here we recall the key result in their paper.
\subsection{Snake graph and perfect matchings}
Let $\Gamma=\left\{\tau_1,\tau_2,\cdots,\tau_n,\cdots,\tau_m\right\}$ where the $\tau_i$ are internal arcs for $1\leq i\leq n$ and $\tau_i$ are boundary arcs for $n+1\leq i\leq m$. Let $\gamma$ be an arc in $(S,M)$. Let $p_0$ be the starting point of $\gamma$ and $p_{d+1}$ be its endpoint. Assume $\gamma$ crosses $\Gamma$ at $p_1,\cdots,p_d$ in order. Let $\tau_{i_j}$ be the arc in $\Gamma$ containing $p_j$. Let $\Delta_{j-1}$ and $\Delta_{j}$ be the two triangles in $\Gamma$ on either side of $\tau_{i_j}$.
\par For each $p_j$, we associate a $tile$ $G(p_j)$ as follows. Define $\Delta^j_1$ and $\Delta^j_2$ to be two triangles with edges labeled as in $\Delta_{j-1}$ and $\Delta_{j}$, further, the orientations of $\Delta^j_1$ and $\Delta^j_2$ both agree with those of $\Delta_{j-1}$ and $\Delta_{j}$ if $j$ is odd; the orientations of $\Delta^j_1$ and $\Delta^j_2$ both disagree with those of $\Delta_{j-1}$ and $\Delta_{j}$ if $j$ is even. We glue $\Delta^j_1$ and $\Delta^j_2$ at the edge labeled $\tau_{i_j}$, so that the orientation of $\Delta^j_1$ and $\Delta^j_2$ both either agree or disagree with those of $\Delta_{j-1}$ and $\Delta_{j}$. We call the edge labeled $\tau_{i_j}$ the diagonal of $G(p_j)$.
\par The two arcs $\tau_{i_j}$ and $\tau_{i_{j+1}}$ form two edges of the triangle $\Delta$. Denote the third edge of $\Delta$ by $\tau_{[\gamma_j]}$. After gluing the tiles $G(p_j)$ and $G(p_{j+1})$ at the edge labeled $\tau_{[\gamma_j]}$ for $1\leq j\leq d-1$ step by step, we obtain a graph, denoted by $\overline{G_{\Gamma,\gamma}}$. Let $G_{\Gamma,\gamma}$ be the graph obtained from $\overline{G_{\Gamma,\gamma}}$ by removing the diagonal of each tile.
\par In particular, when $\gamma\in \Gamma$, let $G_{\Gamma,\gamma}$ be the graph with only one edge
labeled $\gamma$.
\begin{definition}
	\begin{enumerate}
		\item We call a snake graph $G$ with tiles $G_1,\cdots,G_s$ a $zigzag$ if for all $i$ either $G_i$ is glued on top of $G_{i-1}$ and $G_{i+1}$ is glued to the right of $G_i$ or if $G_i$ is glued to the right of $G_{i-1}$ and $G_{i+1}$ is glued on top of $G_i$.
		\item We call a snake graph $G$ with tiles $G_1,\cdots,G_s$ a $straight \ piece$ if for all $i$ either $G_i$ is glued to the right of $G_{i-1}$ and $G_{i+1}$ is glued to the right of $G_i$ or if $G_i$ is on top of $G_{i-1}$ and $G_{i+1}$ is glued on top of $G_i$.
	\end{enumerate}
\end{definition}
\begin{definition}
	A $perfect$ $matching$ of a graph $G$ is a subset $P$ of the edges of $G$ such that each vertex of $G$ is incident to exactly one edge of $P$. Denote the set of all perfect matchings of $G$ by $\mathcal{P}(G)$.
\end{definition}
\begin{lemma}[\cite{Huang2021AnEF}]
	Let $G_i$ and $G_{i+1}$ be two consecutive tiles of $G_{\Gamma,\gamma}$ sharing the same edge a. If $b$ and $c$ are edges of $G_i$ and $G_{i+1}$, respectively, which are incident to $a$, then $b$ and $c$ can not simultaneously be in a perfect matching of $G$.
\end{lemma}
\begin{definition}
	Let $a_1$ and $a_2$ be the two edges of $G_{\Gamma,\gamma}$ which lie in the counterclockwise direction from the diagonal of $G(p_1)$. Then the $minimal$ $matching$ $P_-(G_{\Gamma,\gamma})$ is defined as the unique matching which contains only boundary edges and does not contain edges $a_1$ or $a_2$. The $maximal$ $matching$ $P_+(G_{\Gamma,\gamma})$ is the other matching with only boundary edges.
\end{definition}
\begin{lemma}[\cite{Huang2021AnEF}]
	Let $a$ be an edge of the tile $G(p_j)$. If $a$ is in the maximal/minimal perfect matching of $G_{\Gamma,\gamma}$, then $a$ lies in the counterclockwise/clockwise direction from the diagonal of $G(p_j)$ when $j$ is odd and lies in the clockwise/counterclockwise direction from the diagonal of $G(p_j)$ when $j$ is even.
\end{lemma}
\begin{remark}
	This property is actually a standard criterion for determining whether a boundary edge lies in the minimal or maximal matching. For example, if $a$ is a boundary edge lying in the clockwise direction from the diagonal of $G(p_j)$ when $j$ is odd, then $a$ must belong to the minimal matching.
\end{remark}
\begin{definition}
	A perfect matching $P$ can $twist$ on a tile $G(p)$ if $G(p)$ has two edges belonging to $P$. In such a case we define the $twist$ $\mu_p(P)$ of $P$ on $G(p)$ to be the perfect matching obtained from $P$ by replacing the edges in $G(p)$ by the remaining two edges.
\end{definition}
\begin{lemma}
	For any two perfect matchings $P,Q\in G_{\Gamma,\gamma}$, $Q$ can be obtained from $P$ by a sequence of twists.
\end{lemma}
\subsection{Quantum Laurent expansions}
For any $1\leq s\leq d$, denote by $G^+_{p_s}$ the subgraph of $G_{\Gamma,\gamma}$ formed by the tiles $G(p_{s+1}),\cdots,G(p_d)$ and $G^-_{p_s}$ the subgraph of $G_{\Gamma,\gamma}$ formed by the tiles $G(p_1),\cdots,G(p_{s-1})$.
\begin{definition}\label{d1}
	\begin{enumerate}
		\item Define
		\begin{center}
			$m^{\pm}_{p_s}(\tau_{i_s},\gamma)$=Number of diagonals labeled $\tau_{i_s}$ of $G^\pm_{p_s}$.
		\end{center}
		\item For any perfect matching $P$ that can twist on the tile $G(p_s)$, define
		\begin{center}
			$n^{\pm}_{p_s}(\tau_{i_s},P)$=Number of edges labeled $\tau_{i_s}$ of $P\cap E(G_{p_s}^\pm)$.
		\end{center}
		where $E(G_{p_s}^\pm)$ is the edge set of $G^\pm_{p_s}$.
	\end{enumerate}
\end{definition}
Fix $s$. Assume that $a_{1_s}$, $a_{4_s}$, $\tau_{i_s}$ and $a_{2_s}$, $a_{3_s}$, $\tau_{i_s}$ are triangles in $\Gamma$ such that $a_{1_s}$, $a_{3_s}$ are clockwise to $\tau_{i_s}$ and $a_{2_s}$, $a_{4_s}$ are counterclockwise to $\tau_{i_s}$. Note that $a_{i_s}$, for $i\in 1,2,3,4$ may be boundary arcs.
\begin{definition}\label{omega}
	Suppose that $P$ can twist on $G(p_s)$. If the edges labeled $a_{2_s}$, $a_{4_s}$ of $G(p_s)$ are in $P$, define
	\begin{equation*}
		\Omega(p_s,P)=n^+_{p_s}(\tau_{i_s},P)-m^+_{p_s}(\tau_{i_s},\gamma)-n^-_{p_s}(\tau_{i_s},P)+m^-_{p_s}(\tau_{i_s},\gamma)
	\end{equation*}
	otherwise, define
	\begin{equation*}
		\Omega(p_s,P)=-[n^+_{p_s}(\tau_{i_s},P)-m^+_{p_s}(\tau_{i_s},\gamma)-n^-_{p_s}(\tau_{i_s},P)+m^-_{p_s}(\tau_{i_s},\gamma)]
	\end{equation*}
\end{definition}
In the following, we aim to establish a connection between cluster algebras and surfaces. Here is a key result in \cite{10.1007/s11511-008-0030-7}, \cite{Huang2021AnEF}, \cite{mandel2023braceletsbasesthetabases}.
\begin{theorem}
	There is a bijection $\gamma\leftrightarrow X_\gamma$ between the internal arcs in $(S,M)$ and the cluster variables in $\mathcal{A}$. In particular, the arcs in $\Gamma$ are in one-to-one correspondence with the cluster variables in the initial seed. That is, $X_{\tau_i}=X^{e_i}$.
\end{theorem}
\begin{lemma}
	Let $P\in \mathcal{P}(G_{\Gamma,\gamma})$. The $symmetric \ difference$
	\begin{equation*}
		(P_-(G_{\Gamma,\gamma})\cup P)\backslash(P_-(G_{\Gamma,\gamma})\cap P)
	\end{equation*}
	is a set of boundary edges of a (possibly disconnected) subgraph of $G_{\Gamma,\gamma}$, which is a union of cycles.
\end{lemma}
\begin{definition}
	\begin{enumerate}
		\item If $\gamma$ is an arc and $\tau_{i_1},\cdots,\tau_{i_d}$ is the sequence of arcs in $\Gamma$ which $\gamma$ crosses, then we define the $crossing$ $monomial$ of $\gamma$ with respect to $\Gamma$ to be
		\begin{equation*}
			c(\gamma,\Gamma)=\prod_{j=1}^{d}x_{\tau_{i_j}}
		\end{equation*}
		\item Let $P\in \mathcal{P}(G_{\Gamma,\gamma})$. If the edges of $P$ are labeled $\tau_{j_1},\dots,\tau_{j_r}$, then we define the $weight$ $w(P)$ of $P$ to be
		\begin{equation*}
			w(P)=x_{\tau_{j_1}}\cdots x_{\tau_{j_r}}
		\end{equation*}
	\end{enumerate}
	Note that $x_\alpha=1$ if $\alpha$ is a boundary arc.
\end{definition}
\begin{definition}
	The cluster monomial $x(P)$ associated with $P$ is
	\begin{equation*}
		x(P)=\frac{w(P)}{c(\gamma,\Gamma)}
	\end{equation*} 
\end{definition}
\begin{remark}
	For any $P\in \mathcal{P}(G_{\Gamma,\gamma})$, we have an element $x(P)=x^{a(P)}$ where $a(P)\in \mathbb{Z}^n$. Let $X(P)=X^{a(P)}$ be the unique quantum cluster monomial in $\mathcal{A}_q$.
\end{remark}
\begin{theorem}\label{mapv}
	There exists a unique valuation map $v:\mathcal{P}(G_{\Gamma,\gamma})\rightarrow \mathbb{Z}$ such that
	\begin{enumerate}
		\item[(a)] (initial conditions) $v(P_-(G_{\Gamma,\gamma}))=v(P_+(G_{\Gamma,\gamma}))=0$
		\item[(b)] (iterated relation) If $P\in \mathcal{P}(G_{\Gamma,\gamma})$ can twist on a tile $G(p)$, then
		\begin{equation*}
			v(P)-v(\mu_p(P))=\Omega(p,P)
		\end{equation*}
	\end{enumerate}
\end{theorem}
\begin{theorem}[\cite{Huang2021AnEF}]
	If $\gamma$ is an arc in $(S,M)$, then the quantum Laurent expansion formula of $X_\gamma$ with respect to $\Gamma$ is
	\begin{equation*}
		X_\gamma=\sum_{P\in \mathcal{P}(G_{\Gamma,\gamma})}q^{\frac{v(P)}{2}}X(P)
	\end{equation*}
\end{theorem}
\subsection{Snake graphs and string modules} The theory of snake graphs is closely related to string modules (see \cite{Brstle2011AMI}, \cite{CANAKCI2021102094}). Here we recall some results in their papers. More specifically, there exists a bijection between string modules and snake graphs, up to a certain equivalence.
\begin{itemize}
	\item (From strings to snake graphs) Given a string $w$. If $w$ is a vertex, then the corresponding snake graph is a single tile. Otherwise, suppose $w=a_1\cdots a_n$ with $a_i\in \left\{\rightarrow,\leftarrow\right\}$.
	\begin{equation*}
		x_1\stackrel{a_1}{-}x_2\stackrel{a_2}{-}\cdots x_s\stackrel{a_s}{-}x_{s+1}
	\end{equation*}
	Then, for each vertex $x_i$, there is a tile $G(i)$ with diagonal $x_i$. If $a_1=\rightarrow$, then $G(2)$ is to the right of $G(1)$; otherwise, it is above $G(1)$. If $a_{i+1}$ and $a_{i}$ have the same direction, then $G(i),G(i+1),G(i+2)$ are zigzag. Otherwise, they are straight. Gluing all the $G(i)$, we get the snake graph $G(w)$.
	\item (From snake graphs to strings) For any snake graph $G$ with tiles $G(1),\cdots,G(s)$, the string $w$ is the concatenation of the $a_i$ connecting the diagonals of $G(i)$ and $G(i+1)$. If $G(2)$ is to the right of $G(1)$, then $a_1=G(1)\rightarrow G(2)$. Otherwise, $a_1=G(1)\leftarrow G(2)$. If $G(i),G(i+1),G(i+2)$ are zigzag, then $a_i$ and $a_{i+1}$ have the same direction.  Otherwise, they are inverses to each other.
\end{itemize}
In addition, if there exists an arc $\gamma$ in $(S,M)$ whose string is $w$, the snake graph $G(w)$ and $G_{\Gamma,\gamma}$ coincide. Moreover, this construction induces a bijection between the canonical submodules of $M(w)$ and perfect matchings of $G(w)$. More specifically, any perfect matching $P$ determines an index set in the following way: the symmetric difference $(P_-(G_{w})\cup P)\backslash(P_-(G_{w})\cap P)$ forms a set of enclosed tiles of $G(w)$. The index set corresponding to these tiles is the index set of the canonical submodule $M(P)$.
\begin{theorem}[\cite{CANAKCI2021102094}]
	The map $P\mapsto M(P)$ is a bijection from $\mathcal{P}(G(w))$ to $CS(M(w))$. 
\end{theorem}
Under this bijection, we identify perfect matchings and canonical submodules. Moreover, by the property of the $g$-vector, $x(P)$ has the following factorization
\begin{equation*}
	x(P)=x(P_-(G(w)))x^{B_\Gamma e}
\end{equation*}
where $e=(e_1,e_2,\dots,e_n)$, and $e_i$ is the number of tiles in the symmetric difference whose diagonal is labeled $\tau_i$. Also, $e$ is the dimension vector of $M(P)$. If $\gamma$ is the arc with string $w$, then the $g$-vector of $\gamma$ is given by
\begin{equation*}
	x^{g(\gamma)}=x(P_-(G(w)))
\end{equation*}
The $g$-vector is the index of $M(w)$.
\begin{equation*}
	g(\gamma)=ind_\Gamma(M(w))
\end{equation*}
Denote $x^\Gamma(P)=x^{ind_{\Gamma}(M(w))+B_\Gamma dimP}$, and let $X^\Gamma(P)$ be the corresponding element in the quantum cluster algebra, then
\begin{equation*}
	x(P)=x^\Gamma(P)
\end{equation*}
and the quantum version $X(P)=X^\Gamma(P)$.

\section{MODULE INTERPRETATION}
In this section, we provide a module interpretation of the map $v$ in \ref{mapv}. The key observation is the following:
\begin{itemize}
	\item The snake graphs correspond to the string modules.
	\item The perfect matchings correspond to canonical submodules.
	\item The number of edges labeled $\tau_k$ corresponds to the index of canonical submodules.
\end{itemize}
For any arc $\gamma$ which does not belong to $\Gamma$, let $w=a_1\cdots a_s$ be its string and $G(w)$ be its snake graph. The quantum cluster variable corresponding to $\gamma$ is
\begin{equation*}
	X_{\gamma}=\sum_{P\in \mathcal{P}(G(w))}q^{\frac{v(P)}{2}}X(P)
\end{equation*}
Denote the vertices of $w$ by $x_1,x_2,\cdots,x_{s+1}$. Let $P$ be a perfect matching of $G(w)$. The minimal perfect matching of $G(w)$ is denoted by $P_-$. The symmetric difference provides a collection of enclosed tiles, which can give an index subset of $\left\{1,2,\dots,s\right\}$. Under the equivalence above, we still denote the index set of the canonical submodule $M(P)$ by $P$. First we need to provide an algorithm for the number of edges.
\par Let $\tau_k$ be a vertex of $Q$. Then after choosing a suitable label the subquiver of the neighborhood of $\tau_k$ and the triangle in the surface are shown in Figure \ref{quiver}. Note that some vertices may be frozen.
\begin{figure}[htb]
	\centering
	\includegraphics[scale=0.2]{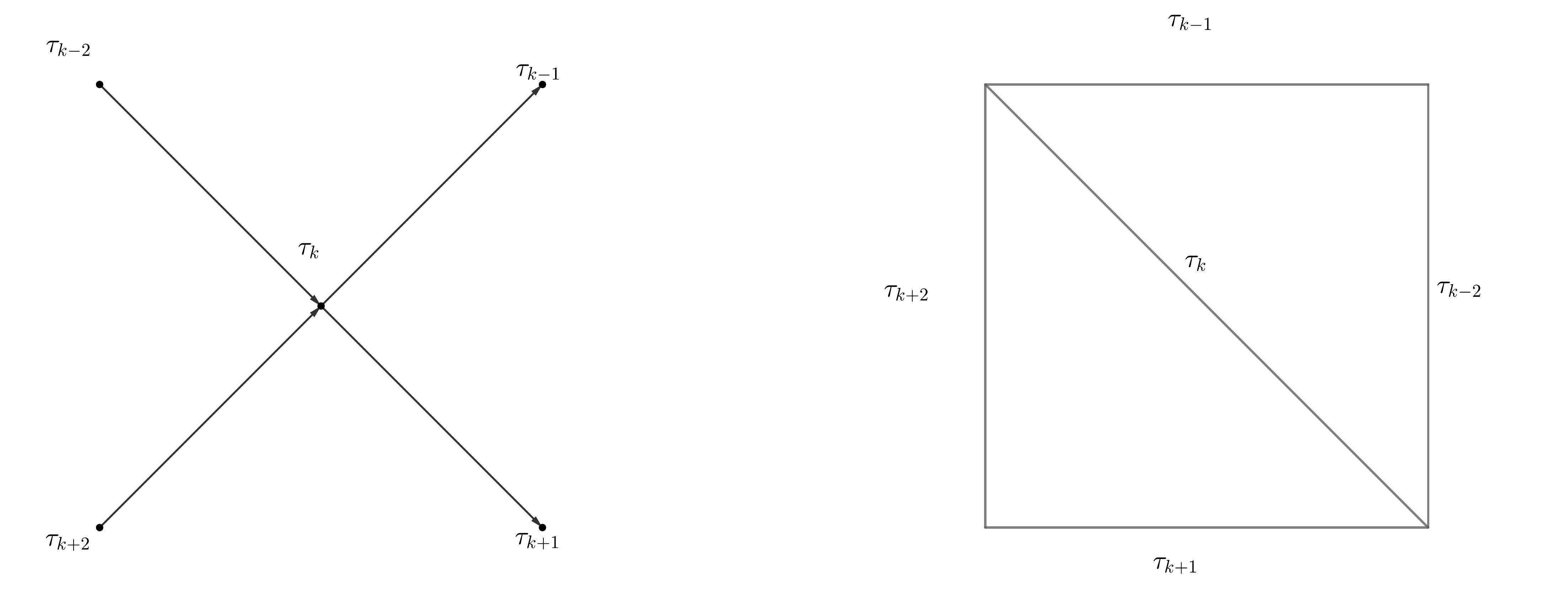}
	\caption{Neighborhood of $\tau_k$}
	\label{quiver}
\end{figure}

\begin{lemma}
	If there exists some $j$ with $2\leq j\leq s-1$ satisfying $x_j=\tau_k$ and $a_{j-1}$ and $a_j$ are both inverse, define two quantities to count the number of edges labeled $\tau_k$ in $E(G(j+1))\cap P$ and $E(G(j-1))\cap P$.
	\begin{equation*}
		n^+(\tau_k,j,P)=\left\{
		\begin{aligned}
			1 & & &if \ j+1\in P \\
			0 & & &otherwise
		\end{aligned}
		\right.  
	\end{equation*}
	\begin{equation*}
		n^-(\tau_k,j,P)=\left\{
		\begin{aligned}
			0 & & &if \ j-1\in P \\
			1 & & &otherwise
		\end{aligned}
		\right.  
	\end{equation*}
\end{lemma}
\begin{proof}
	Without loss of generality, assume $x_{j-1}=\tau_{k-1}$, $x_{j+1}=\tau_{k+2}$. The snake graph of $G(j-1),G(j),G(j+1)$ is shown in Figure \ref{2inv}. We want the notion $n^+$ (resp.$n^-$) to be the number of edges labeled $\tau_k$ belonging to $P$ on the tiles with diagonal $x_{j+1}$ (resp. $x_{j-1}$), and we only focus on the edges incident to the diagonal. We illustrate this definition when $j$ is odd(left) and the other case is similar. In this case, the blue edge always belongs to $P_-$ and the red edge cannot belong to $P_-$. Then $n^+$ just depends on whether the red edge belongs to $P$. According to the definition of symmetric difference, the red edge is in $P$ if and only if $j+1\in P$. For the same reason, the blue edge incident to $\tau_k$ belongs to $P$ if and only if $j-1\notin P$.
\end{proof}
\begin{remark}
	The quantities defined here are used to count the number of edges in an equivalence class defined in \cite{Huang2021AnEF}. There the edges incident to the same diagonal which in addition all have the same label are equivalent. So the number in this equivalence class is $n^++n^-$. The procedure requiring the count of rightward or leftward steps corresponds exactly to the definition of $G^+_{s}$ or $G^-_{s}$, respectively.
\end{remark}
\begin{figure}[htbp]
	\centering
	\includegraphics[scale=0.1]{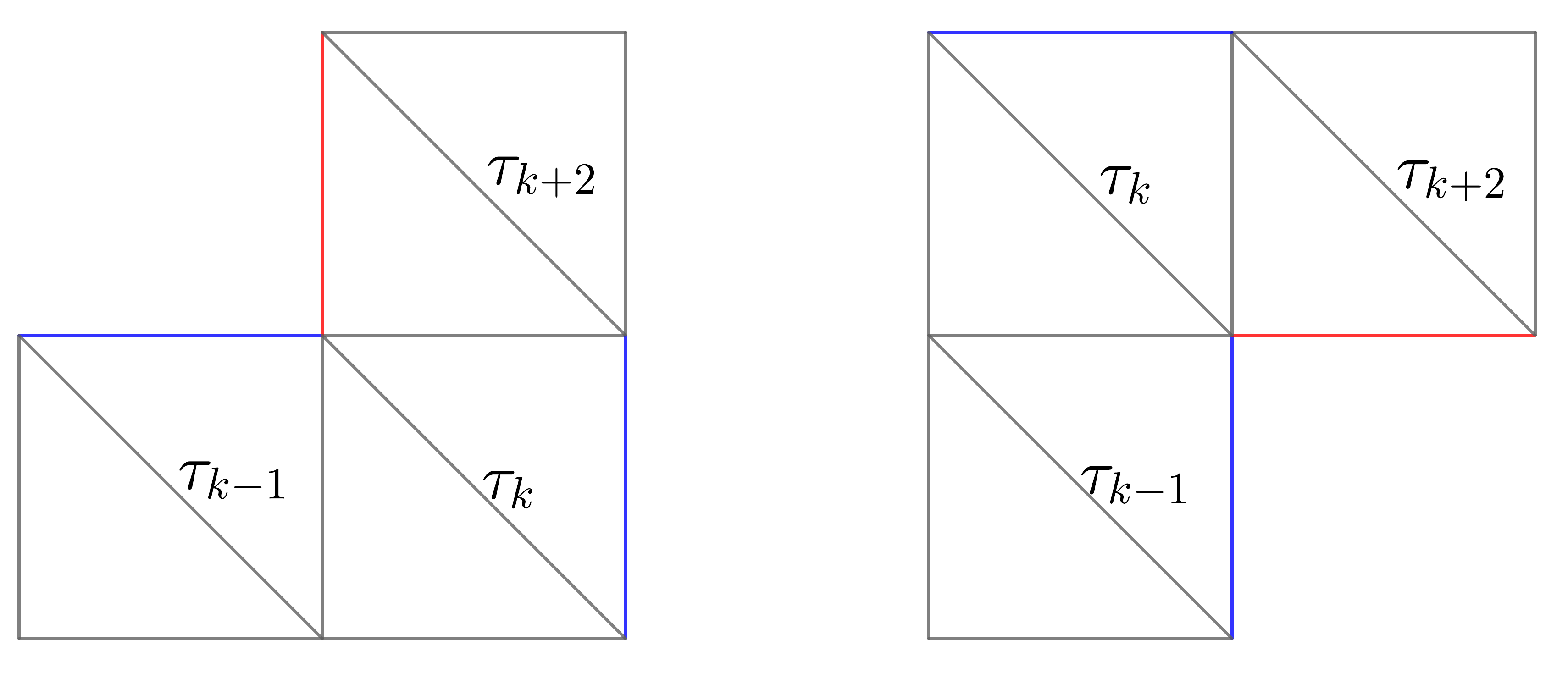}
	\caption{$a_{j-1}$ and $a_j$ both are inverse}
	\label{2inv}
\end{figure}

\begin{lemma}
	If there exists some $j$ with $2\leq j\leq s-1$ satisfying $x_j=\tau_k$ and $a_{j-1}$ and $a_j$ are both direct, define two quantities to count the number of edges labeled $\tau_k$ in $E(G(j+1))\cap P$ and $E(G(j-1))\cap P$.
	\begin{equation*}
		n^+(\tau_k,j,P)=\left\{
		\begin{aligned}
			0 & & &if \ j+1\in P \\
			1 & & &otherwise.
		\end{aligned}
		\right.  
	\end{equation*}
	\begin{equation*}
		n^-(\tau_k,j,P)=\left\{
		\begin{aligned}
			1 & & &if \ j-1\in P \\
			0 & & &otherwise
		\end{aligned}
		\right.  
	\end{equation*}
\end{lemma}
\begin{proof}
	Without loss of generality, assume $x_{j-1}=\tau_{k-2}$, $x_{j+1}=\tau_{k+1}$. The snake graph of $G(j-1),G(j),G(j+1)$ is shown in Figure \ref{2dir}. We want the notion $n^+$ (resp.$n^-$) to be the number of edges labeled $\tau_k$ belonging to $P$ on the tiles with diagonal $x_{j+1}$ (resp. $x_{j-1}$), and we only focus on the edges incident to the diagonal. We illustrate this definition when $j$ is odd (right) and the other case is similar. In this case, the blue edge always belongs to $P_-$ and the red edge cannot belong to $P_-$. Then $n^+$ just depends on whether the blue edge belongs to $P$. According to the definition of symmetric difference, the blue edge is in $P$ if and only if $j+1\notin P$. For the same reason, the red edge incident to $\tau_k$ belongs to $P$ if and only if $j-1\in P$.
\end{proof}
\begin{figure}[htbp]
	\centering
	\includegraphics[scale=0.1]{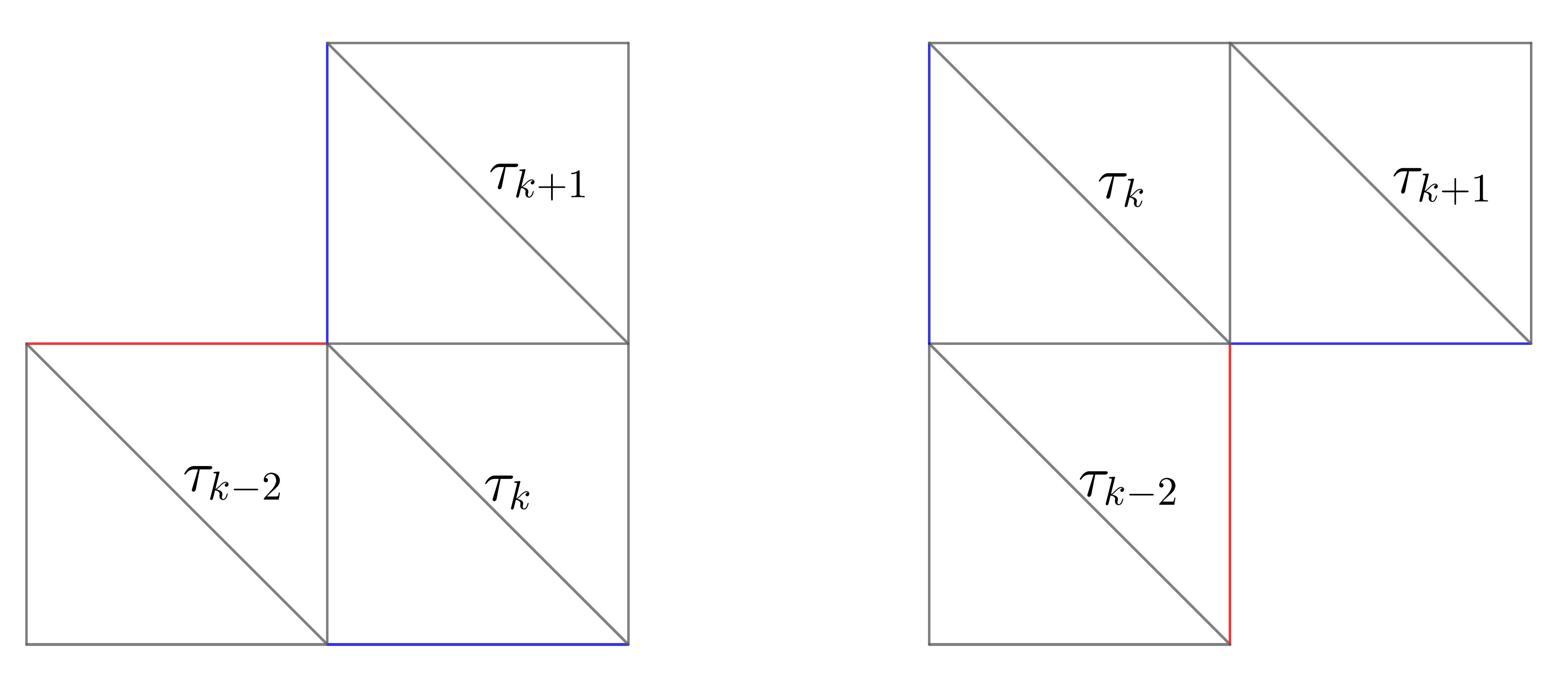}
	\caption{$a_{j-1}$ and $a_j$ both are direct}
	\label{2dir}
\end{figure}

\begin{lemma}
	If there exists some $j$ with $2\leq j\leq s-1$ satisfying $x_j=\tau_k$ and $a_{j-1}$ is inverse and $a_j$ is direct, define two quantities to count the number of edges labeled $\tau_k$ in $E(G(j+1))\cap P$ and $E(G(j-1))\cap P$.
	\begin{equation*}
		n^+(\tau_k,j,P)=\left\{
		\begin{aligned}
			0 \ \ \ &if\ \ j+1\in P \\
			1 \ \ \ &otherwise
		\end{aligned}
		\right.
	\end{equation*}
	\begin{equation*}
		n^-(\tau_k,j,P)=\left\{
		\begin{aligned}
			0 \ \ \ &if\ \ j-1\in P \\
			1 \ \ \ &otherwise
		\end{aligned}
		\right.
	\end{equation*}
\end{lemma}
\begin{proof}
	Without loss of generality, assume $x_{j-1}=\tau_{k-1}$, $x_{j+1}=\tau_{k+1}$. The snake graph of $G(j-1),G(j),G(j+1)$ is shown in Figure \ref{invdir}. The red edge always belongs to $P_-$ and $n^+$ just depends on whether the blue edge belongs to $P$. According to the definition of symmetric difference, the blue edge is in $P$ if and only if $j+1\notin P$. The description for $n^-$ is true for the same reason.
\end{proof}
\begin{figure}[htbp]
	\centering
	\includegraphics[scale=0.1]{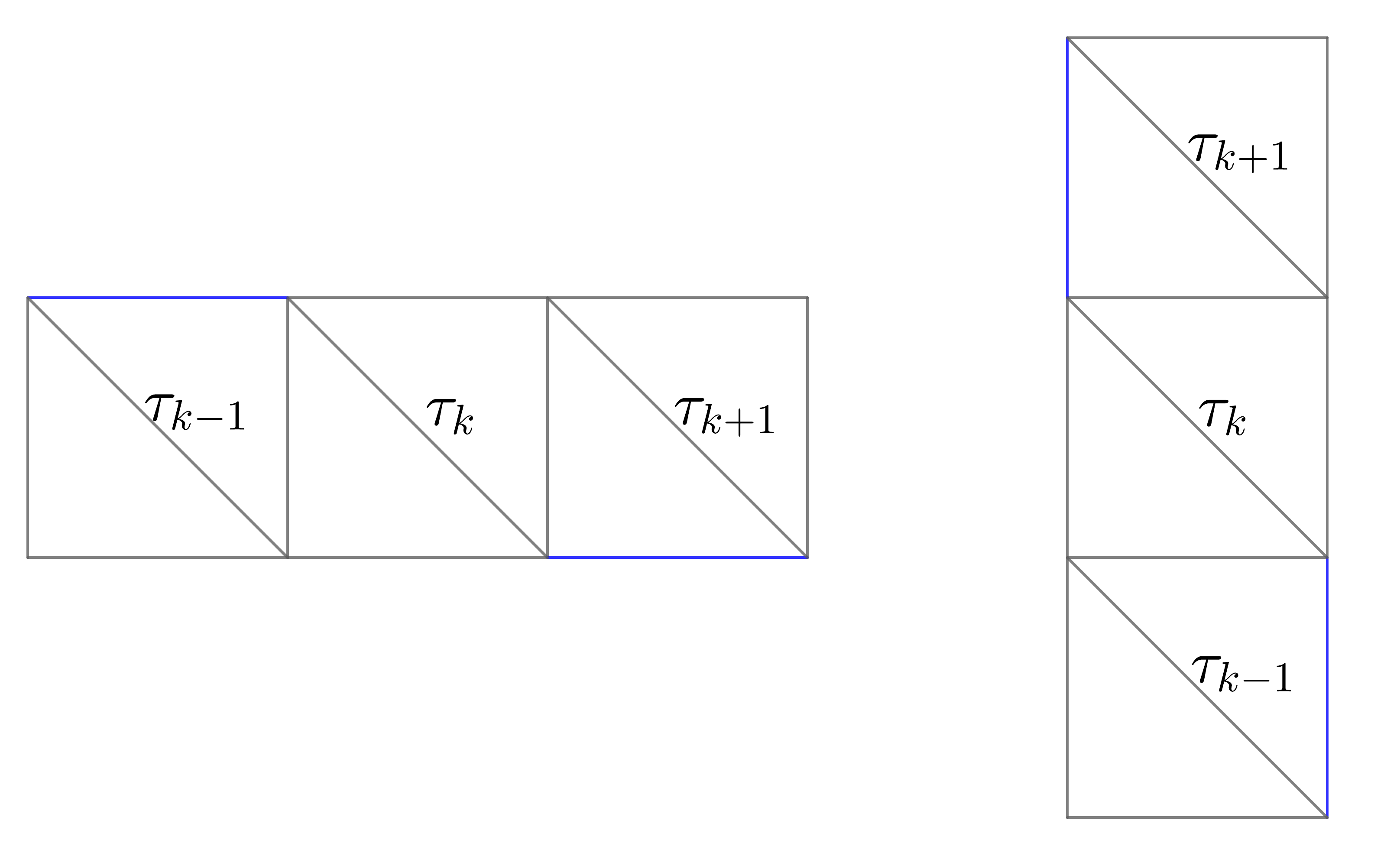}
	\caption{$x_{j-1}\leftarrow x_{j}\rightarrow x_{j+1}$}
	\label{invdir}
\end{figure}

\begin{lemma}
	If there exists some $j$ with $2\leq j\leq s-1$ satisfying $x_j=\tau_k$ and $a_{j-1}$ is direct and $a_j$ is inverse, define two quantities to count the number of edges labeled $\tau_k$ in $E(G(j+1))\cap P$ and $E(G(j-1))\cap P$.
	\begin{equation*}
		n^+(\tau_k,j,P)=\left\{
		\begin{aligned}
			1 \ \ \ &if\ \ j+1\in P \\
			0 \ \ \ &otherwise
		\end{aligned}
		\right.
	\end{equation*}
	\begin{equation*}
		n^-(\tau_k,j,P)=\left\{
		\begin{aligned}
			1 \ \ \ &if\ \ j-1\in P \\
			0 \ \ \ &otherwise
		\end{aligned}
		\right.
	\end{equation*}
\end{lemma}
\begin{proof}
	Without loss of generality, assume $x_{j-1}=\tau_{k-2}$, $x_{j+1}=\tau_{k+2}$. The snake graph of $G(j-1),G(j),G(j+1)$ is shown in Figure \ref{dirinv}. The blue edges always belong to $P_-$ and the red edges do not. $n^+$ just depends on whether the red edge belongs to $P$. According to the definition of symmetric difference, the red edge is in $P$ if and only if $j+1\in P$. The description for $n^-$ is true for the same reason.
\end{proof}
\begin{figure}[htbp]
	\centering
	\includegraphics[scale=0.1]{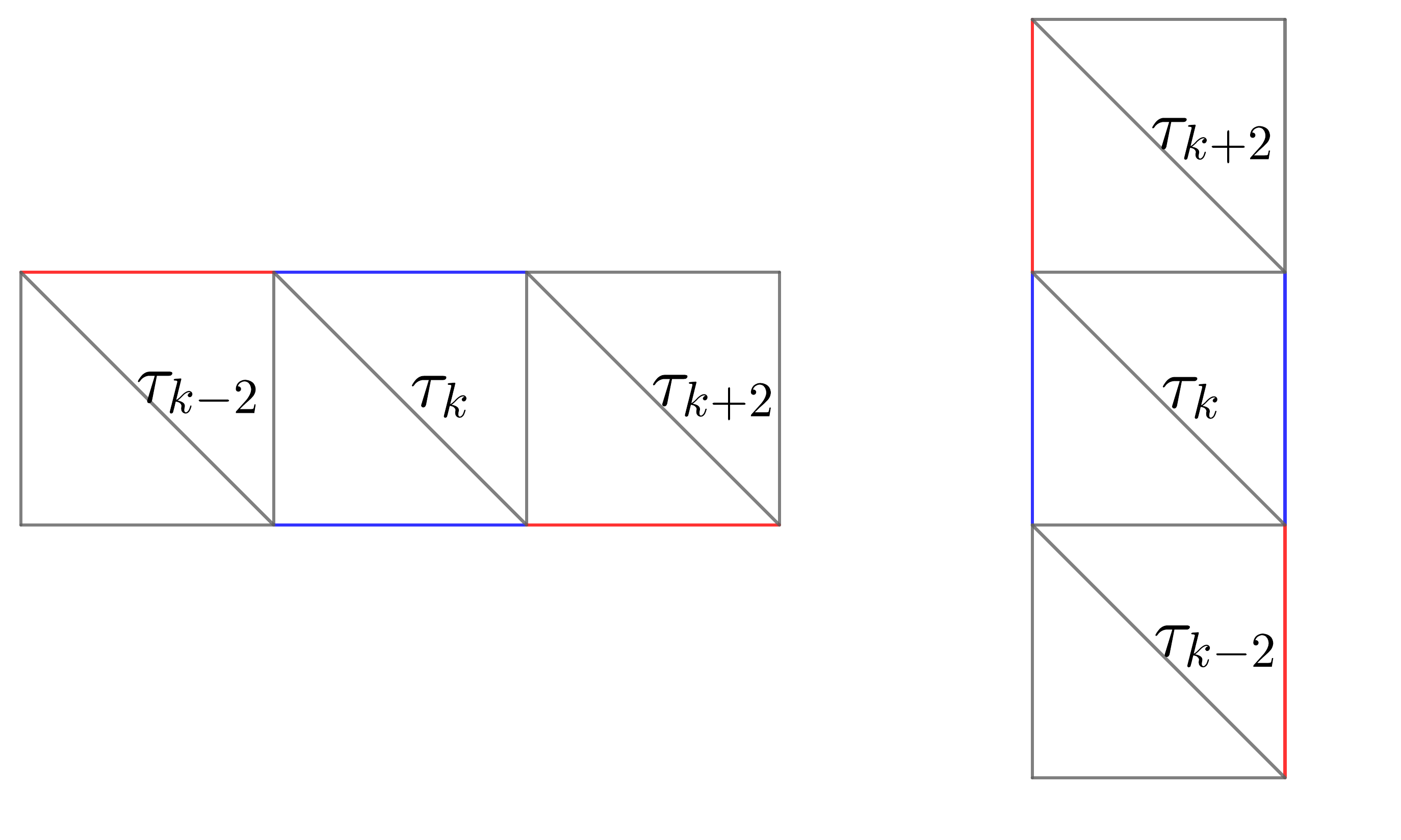}
	\caption{$x_{j-1}\rightarrow x_{j}\leftarrow x_{j+1}$}
	\label{dirinv}
\end{figure}

\begin{lemma}
	If there exists some $j$ satisfying $\left\{x_j,x_{j+1}\right\}=\left\{\tau_{k-1},\tau_{k-2}\right\}$ or $\left\{\tau_{k+1},\tau_{k+2}\right\}$, let $n(\tau_k,j,P)=1$ if exactly one of $j$ and $j+1$ belongs to $P$, otherwise $n(\tau_k,j,P)=0$. Then $n(\tau_k,j,P)$ counts the number of edges in $P$ labeled $\tau_k$ in these two tiles.
\end{lemma}
\begin{proof}
	Without loss of generality, assume $x_{j}=\tau_{k-1}$, $x_{j+1}=\tau_{k-2}$. The snake graph of $G(j-1),G(j),G(j+1)$ is shown in Figure \ref{2inc}. The blue edges do not belong to $P_-$. $n(\tau_k,j,P)$ just depends on whether the blue edge belongs to $P$. According to the definition of symmetric difference, the blue edge is in $P$ if and only if exactly one of $j$ and $j+1$ belongs to $P$.
\end{proof}
\begin{figure}[htbp]
	\centering
	\includegraphics[scale=0.1]{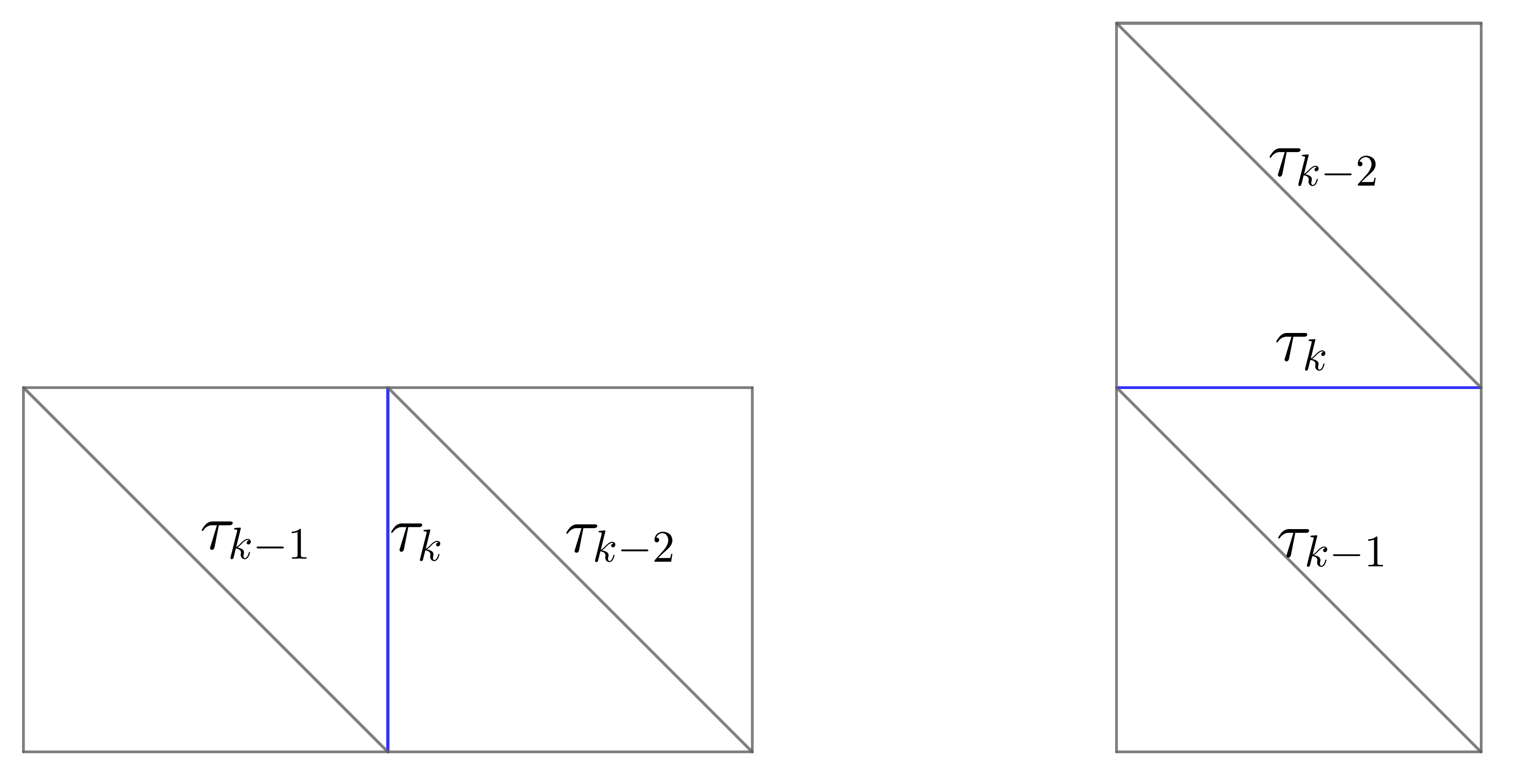}
	\caption{$x_{j}=\tau_{k-1}$, $x_{j+1}=\tau_{k-2}$}
	\label{2inc}
\end{figure}
In the following, we need to consider the first or last tile and count the number of edges labeled $\tau_k$ in these tiles. The strategy is totally the same as above. We show some examples in Figure \ref{firstnot}. For example, if the first tile has diagonal $\tau_{k-1}$ and it cannot be any case above, then the blue edge labeled $\tau_k$ belongs to $P_-$. So whether it belongs to $P$ is equivalent to whether $1$ belongs to $P$. Similarly for the other case. We provide a summary of all results in this case here.
\begin{lemma}
	If $x_1\in \left\{\tau_{k-1},\tau_{k+1}\right\}$ and $x_2\notin \left\{\tau_{k-2},\tau_{k-1},\tau_k,\tau_{k+1},\tau_{k+2}\right\}$, then let $n(\tau_k,1,P)=0$ if $1\in P$ and $n(\tau_k,1,P)=1$ if $1\notin P$. If $x_1\in \left\{\tau_{k-2},\tau_{k+2}\right\}$, then $n(\tau_k,1,P)=1$ if $1\in P$ and $n(\tau_k,1,P)=0$ if $1\notin P$. In both cases, $n(\tau_k,1,P)$ counts the number of edges labeled $\tau_k$ in $P$ in the first tile.
\end{lemma}
Dually, for the last tile
\begin{lemma}
	If $x_s\in \left\{\tau_{k-1},\tau_{k+1}\right\}$ and $x_{s-1}\notin \left\{\tau_{k-2},\tau_{k-1},\tau_k,\tau_{k+1},\tau_{k+2}\right\}$, then $n(\tau_k,s,P)=1$ if $s\notin P$ and $n(\tau_k,s,P)=0$ if $s\in P$. If $x_s\in \left\{\tau_{k-2},\tau_{k+2}\right\}$, then $n(\tau_k,s,P)=0$ if $s\notin P$ and $n(\tau_k,s,P)=1$ if $s\in P$. In both cases, $n(\tau_k,s,P)$ counts the number of edges labeled $\tau_k$ in $P$ in the last tile.
\end{lemma}
\begin{figure}[htbp]
	\centering
	\includegraphics[scale=0.1]{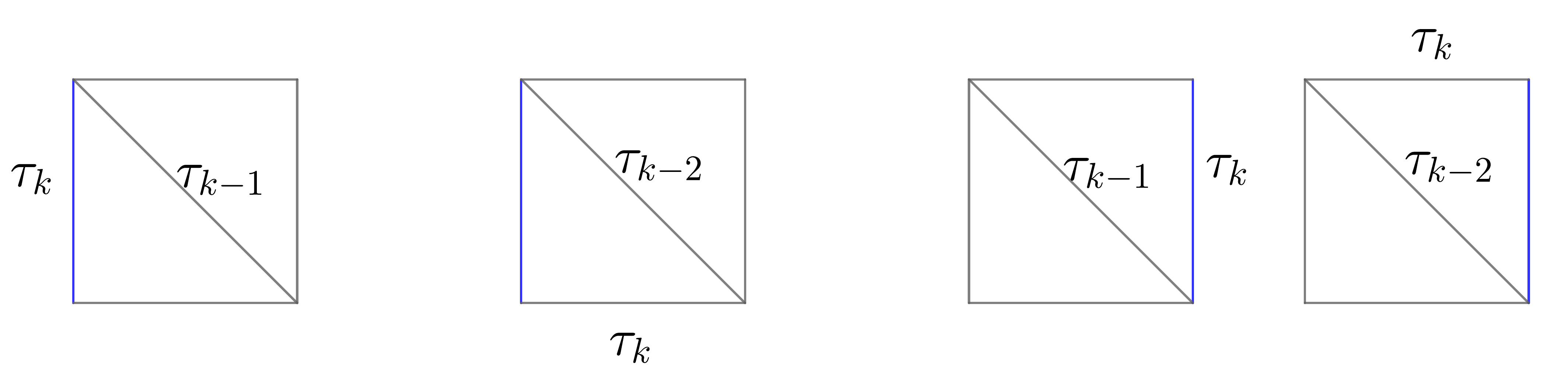}
	\caption{$s$ is odd}
	\label{firstnot}
\end{figure}

In the last case, we need to consider the first or last two tiles and count the number of edges labeled $\tau_k$ in these tiles. We show some examples in Figure \ref{firstis}. For example, if the first tile has diagonal $\tau_{k}$ and the second tile has diagonal $\tau_{k-2}$, then the blue edge belongs to $P_-$ and the red edge cannot belong to $P_-$. So whether the red edge belongs to $P$ is equivalent to whether $2$ belongs to $P$. Similarly for the other case. We provide a summary of all results in this case here.
\begin{lemma}
	If $x_1=\tau_k$ and $x_2\in \left\{\tau_{k-2},\tau_{k+2}\right\}$, then $n^+(\tau_k,1,P)=1$ if $2\in P$ and $n^+(\tau_k,1,P)=0$ if $2\notin P$. If $x_2\in \left\{\tau_{k-1},\tau_{k+1}\right\}$, then $n^+(\tau_k,1,P)=1$ if $2\notin P$ and $n^+(\tau_k,1,P)=0$ if $2\in P$. In both cases, $n(\tau_k,1,P)$ counts the number of edges labeled $\tau_k$ in $P$ in the first two tiles.
\end{lemma}
Dually, for the last two tile 
\begin{lemma}
	If $x_s=\tau_k$ and $x_{s-1}\in \left\{\tau_{k-2},\tau_{k+2}\right\}$, then $n^-(\tau_k,s,P)=1$ if $s-1\in P$ and $n^-(\tau_k,s,P)=0$ if $s-1\notin P$. If $x_{s-1}\in \left\{\tau_{k-1},\tau_{k+1}\right\}$, then $n^-(\tau_k,s,P)=0$ if $s-1\in P$ and $n^-(\tau_k,s,P)=1$ if $s-1\notin P$. In both cases, $n(\tau_k,s,P)$ counts the number of edges labeled $\tau_k$ in $P$ in the last two tiles.
\end{lemma}
\begin{figure}[htbp]
	\centering
	\includegraphics[scale=0.1]{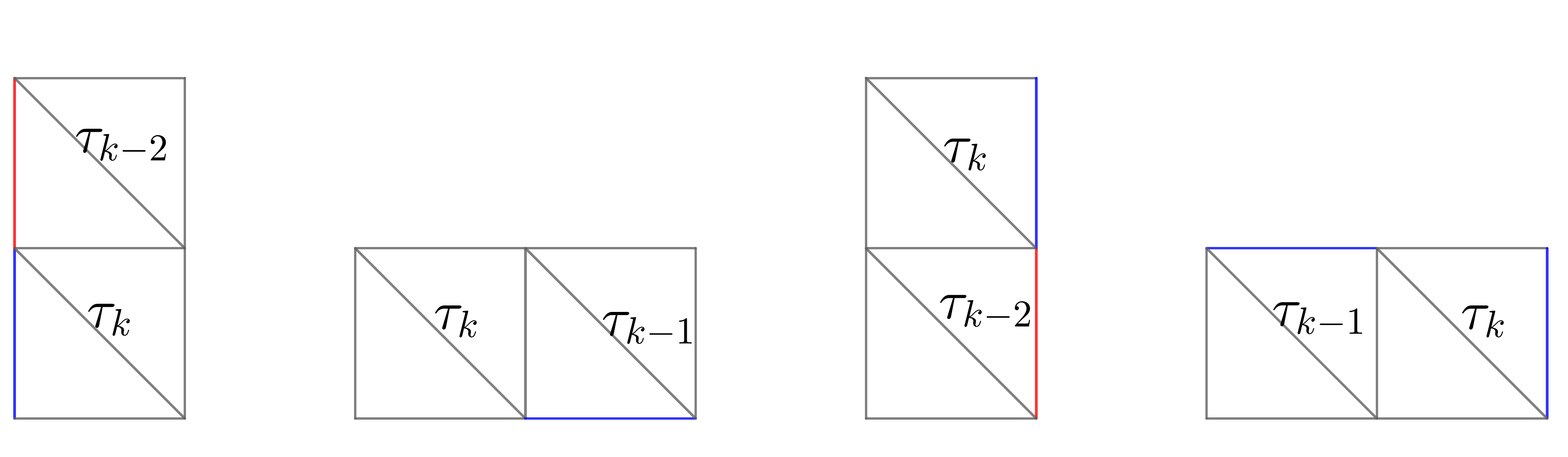}
	\caption{$s$ is odd}
	\label{firstis}
\end{figure}
\begin{remark}
	\begin{enumerate}
		\item For the cases where $n(\tau_k,j,P)$ is not present, let $n(\tau_k,j,P)=n^+(\tau_k,j,P)+n^-(\tau_k,j,P)$.
		\item In \cite{Huang2021AnEF}, the equivalence class of edges labeled $\tau_k$ must be one of the above cases. So $n(\tau_k,j,P)$ is the number of edges belonging to $P$ in the equivalence class near $x_j$.
	\end{enumerate}
\end{remark}
\begin{definition}
	Here we define numbers identical to those in Definition \ref*{d1}.
	\begin{center}
		$M^+(\tau_k,j,w)$=\#$\left\{x_i|x_i=\tau_k, i>j\right\}$ \\
		$M^-(\tau_k,j,w)$=\#$\left\{x_i|x_i=\tau_k, i<j\right\}$ \\
		$N^+(\tau_k,j,P)=n^+(\tau_k,j,P)+\sum_{i>j}n(\tau_k,i,P)$ \\
		$N^-(\tau_k,j,P)=n^-(\tau_k,j,P)+\sum_{i<j}n(\tau_k,i,P)$
	\end{center}
	Then $M^{\pm}(\tau_k,j,w)=m^{\pm}_j(\tau_k,\gamma)$, $N^{\pm}(\tau_k,j,P)=n^{\pm}_j(\tau_k,P)$.
\end{definition}
\begin{definition}
	If $j\in P$, let
	\begin{equation*}
		\Omega'(x_j,P)=[N^+(x_j,j,P)-M^+(x_j,j,w)-N^-(x_j,j,P)+M^-(x_j,j,w)]
	\end{equation*}
	Otherwise, define
	\begin{equation*}
		\Omega'(x_j,P)=-[N^+(x_j,j,P)-M^+(x_j,j,w)-N^-(x_j,j,P)+M^-(x_j,j,w)]
	\end{equation*}
\end{definition}
\begin{definition}
	Let $\gamma$ be an arc in $(S,M)$ with string $w$. Let $v_\gamma:CS(M(w))\rightarrow \mathbb{Z}$ be a map such that
	\begin{enumerate}
		\item $v_\gamma(0)=0$.
		\item If two submodules $N$ and $N'$ satisfy $N=N'\cup {x_j}$ for an index $j$, then
		\begin{equation*}
			v_\gamma(N)-v_\gamma(N')=\Omega'(x_j,N)
		\end{equation*}
	\end{enumerate}
\end{definition}
\begin{theorem}
	The quantum cluster variable associated with $M(w)$ can be expressed by
	\begin{equation*}
		X_{M(w)}=\sum_{N\in CS(M(w))}q^{\frac{v_\gamma(N)}{2}}X^\Gamma(N)
	\end{equation*}
\end{theorem}
\begin{proof}
	The arcs in the surface or their string modules are in bijection with the quantum cluster variables. Note that in Definition \ref{omega}, the edges labeled $a_{2_s}$ and $a_{4_s}$ belonging to $P$ are equivalent to $p_s\in P$. So $\Omega'(x_j,P)=\Omega(j,P)$ and the map $v_\gamma$ has the same initial condition and recursion formula. Therefore, the map $v_\gamma$ is equal to the map $v$ in Theorem \ref{mapv}. For any $P$, we also have $X(P)=X^\Gamma(P)$. Then the formula is true by the bijection between perfect matchings and canonical submodules.
\end{proof}

\section{THE SKEIN ALGEBRA}In this section, we recall some facts about skein algebras from \cite{mandel2023braceletsbasesthetabases}, \cite{Muller2016}.
\par Let $(S,M)$ be a marked surface. A multicurve is an immersion $\phi: C\rightarrow S$ of a compact unoriented 1-manifold $C$ such that the boundary of $C$ maps to $M$, but no interior points of $C$ map to $\partial S$ or $M$. A curve is a connected multicurve. A strand in a multicurve $\phi: C\rightarrow S$ near a point $p$ is a component of the restriction of $\phi$ to an arbitrarily small disc around $p$. For any given arc, we choose two strands which contain the endpoints and do not intersect in $S\backslash M$, and we call these strands the ends of the arc. All multicurves are considered up to homotopy. So we can choose all the self-intersections to be transverse and all interior crossings are between only two strands.
\par A link is a transverse multicurve together with a choice of ordering of the strands at each interior crossing, which means a choice of which strand is \enquote{over} and which is \enquote{under}. Let $K_q=K[q^{\pm \frac{1}{2}}]$. Let $K_q^{Links}(S)$ denote the free $K_q$-module with basis given by the homotopy classes $[C]$ of links $C$ in $S$.
Then the skein module $SK_q(S)$ is defined to be $K_q^{Links}(S)/R$, where $R$ denotes the module of relations in $K_q^{Links}(S)$ generated by the following ``skein'' relation:
\begin{itemize}
	\item Contractible arcs are equivalent to 0;
	\item A contractible loop is equivalent to $-(q^2+q^{-2})\cdot \emptyset$, where $\emptyset$ denotes the empty link;
	\item The Kauffman skein relation.
\end{itemize}
The skein module $SK_q(S)$ is an associative algebra using the superposition product: if $X$ and $Y$ are links such that $X\cup Y$ has transverse crossings, then $X\cdot Y$ is equal to $q^{\frac{k}{2}}$ times the link $X\cup Y$ in which strands of $X$ always cross over strands of $Y$ at each crossing. The integer $k$ is defined as follows: for each arc $i$, let $\partial_1(i)$ and $\partial_2(i)$ denote the two ends of $i$ (for arbitrary numbering). Then given two arcs $i$, $j$, define 
\begin{align*}
	\Lambda(i,j)=\sum_{a,b\in \left\{1,2\right\}}\left\{
	\begin{aligned}
		0& \ \ \mbox{if $\partial_a(i)$  and  $\partial_b(j)$ have different endpoints}    \\
		1& \ \ \mbox{if  $\partial_a(i)$ is clockwise of  $\partial_b(j)$}          \\
		-1& \ \ \mbox{if  $\partial_a(i)$ is counterclockwise of  $\partial_b(j)$} 
	\end{aligned}
	\right.
\end{align*}
Then $k=\sum \Lambda(i,j)$, where the sum is over all pairs of arcs $i\in X$ and $j\in Y$.
\par This algebra is called the (\textbf{Kauffman}) \textbf{skein algebra} of $S$.
\begin{lemma}[\cite{mandel2023braceletsbasesthetabases}]
	Let $(S,M)$ be an unpunctured surface and $\Gamma$ be a triangulation. Then there is a corresponding skew-symmetric matrix $\Lambda$ satisfying $\Lambda B(\Gamma)=-I$. We have an inclusion
	\begin{equation*}
		\mathcal{A}_q(S,M) \subset SK_q(S)
	\end{equation*}
	Furthermore, this inclusion identifies simple arcs with cluster variables, boundary arcs being identified with the frozen variables. Triangulations correspond with clusters and mutations of seeds correspond to flips of triangulations. If each component of $S$ contains at least two markings, then the inclusion is actually an isomorphism.
\end{lemma}

\section{A MULTIPLICATION FORMULA WITH ONE DIMENSIONAL EXTENSION SPACE}
In this section, we provide a multiplication formula between two string modules with one-dimensional extension space. We explicitly give all the strings of these modules. The extension space was studied in \cite{Canakcı_Pauksztello_Schroll_2021}. 
\begin{definition}[\cite{Canakcı_Pauksztello_Schroll_2021}]
	Let $v$ and $w$ be strings.
	\begin{enumerate}
		\item (Arrow extension) If there exists $a\in Q_1$ such that $u=wa^{-1}v$ is a string, then there is a non-split short exact sequence
		\begin{equation*}
			0\rightarrow M(w)\rightarrow M(u)\rightarrow M(v)\rightarrow 0
		\end{equation*}
		\item (Overlap extension) Suppose that $v=v_Lbma^{-1}v_R$ and $w=w_Ld^{-1}mcw_R$ with $a,b,c,d\in Q_1$ and $m,v_L,v_R,w_L,w_R$ (possibly trivial) strings such that
		\begin{enumerate}
			\item[(i)] if $a=\emptyset$, then $c\neq \emptyset$;
			\item[(ii)] if $b=\emptyset$, then $d\neq \emptyset$;
			\item[(iii)] if $m$ is a trivial string, then $ac\in I$ and $bd\in I$ (whenever they exist, subject to the constraints above).
		\end{enumerate}
		Then there exists a non-split short exact sequence
		\begin{equation*}
			0\rightarrow M(w)\rightarrow M(u)\oplus M(u')\rightarrow M(v)\rightarrow 0
		\end{equation*}
		where $u=v_Lbmcw_R$ and $u'=w_Ld^{-1}ma^{-1}v_R$.
	\end{enumerate}
\end{definition}
\begin{theorem}[\cite{Canakcı_Pauksztello_Schroll_2021}]
	Let $A$ be a gentle algebra, and $v$ and $w$ be strings. The collection of arrow and overlap extensions between $M(w)$ and $M(v)$ form a basis for $Ext^1_A(M(v),M(w))$.
\end{theorem}
\begin{definition}[\cite{CANAKCI20171}]
	Given a string $v$,
	\begin{enumerate}
		\item If $v$ is inverse, then $_hv$ is the trivial string corresponding to $e(v)$. Otherwise, $_hv$ is obtained from $v$ by deleting the first direct arrow in $v$ and the inverse string preceding it.
		\item If $v$ is direct, then $_cv$ is the trivial string corresponding to $e(v)$. Otherwise, $_cv$ is obtained from $v$ by deleting the first inverse arrow in $v$ and the direct string preceding it.
		\item If $v$ is direct, then $v_h$ is the trivial string corresponding to $s(v)$. Otherwise, $v_h$ is obtained from $v$ by deleting the last inverse arrow in $v$ and the direct string succeeding it.
		\item If $v$ is inverse, then $v_c$ is the trivial string corresponding to $s(v)$. Otherwise, $v_c$ is obtained from $v$ by deleting the last direct arrow in $v$ and the inverse string succeeding it.
	\end{enumerate}
\end{definition}

\subsection{Arrow extension}
Let $v$ and $w$ have an arrow extension. Figure \ref{arrow} shows the arrow extension on the surface \cite{CANAKCI20171}. 
In particular, $u_1=wa^{-1}v$ is the string of the middle term of the extension. $u_2$ is the trivial string corresponding to some vertex $i_0$. $u_3= {_c}v$ and $u_4=v_h$.
\begin{lemma}
	$vw=q^{\frac{\Lambda(v,w)-\Lambda(u_1,u_2)}{2}+1}[u_1][u_2]+q^{\frac{\Lambda(v,w)-\Lambda(u_3,u_4)}{2}-1}[u_3][u_4]$
\end{lemma}
\begin{proof}
	By definition
	\begin{align*}
		[v][w]&=q^{\frac{\Lambda(v,w)}{2}}[v\cup w] \\
		[u_1][u_2]&=q^{\frac{\Lambda(u_1,u_2)}{2}}[u_1\cup u_2] \\
		[u_3][u_4]&=q^{\frac{\Lambda(u_3,u_4)}{2}}[u_3\cup u_4]
	\end{align*}
	and the skein relation
	\begin{equation*}
		[v\cup w]=q[u_1\cup u_2]+q^{-1}[u_3\cup u_4]
	\end{equation*}
\end{proof}
\begin{figure}[!htbp]
	\centering
	\includegraphics[scale=0.15]{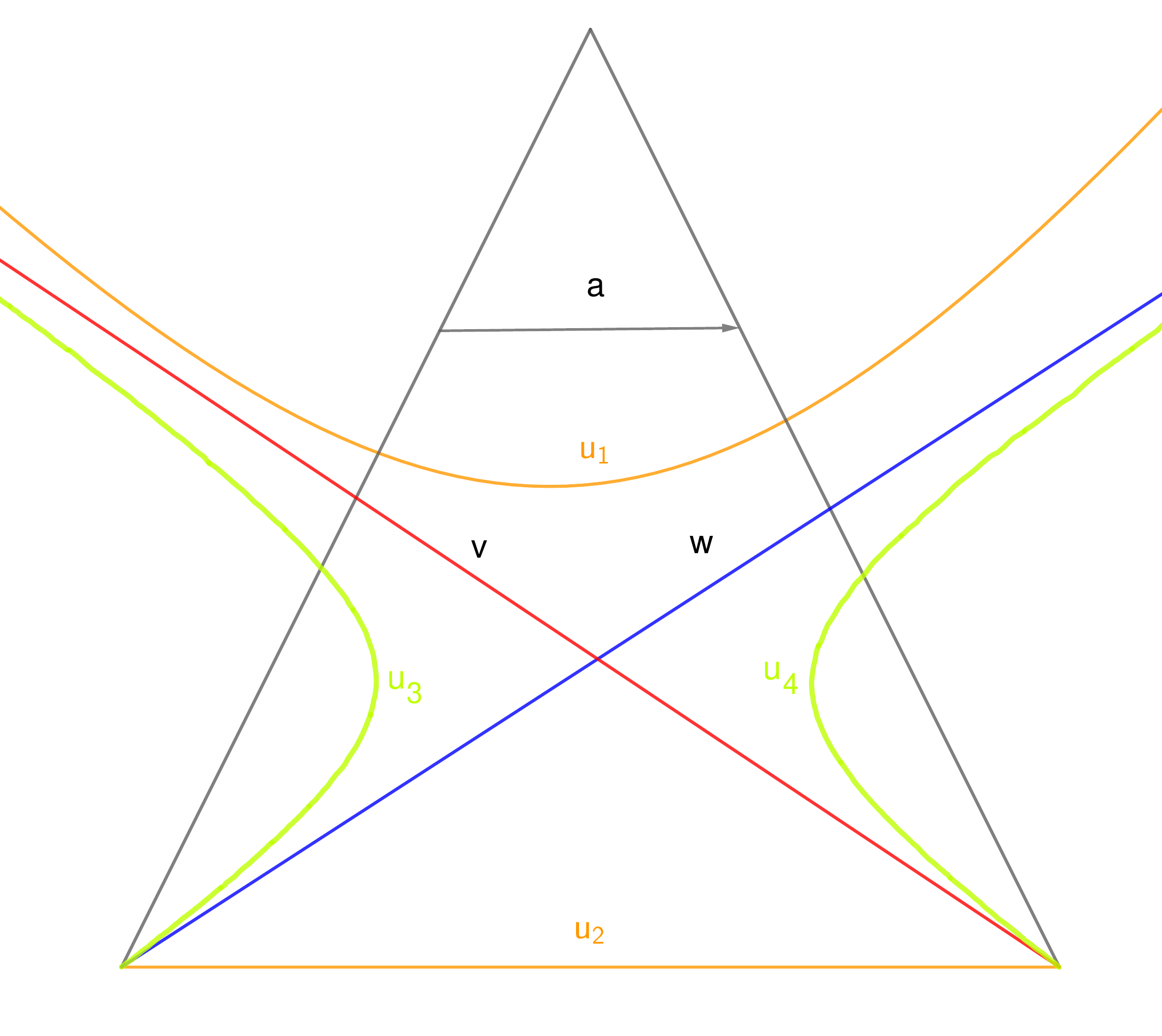}
	\caption{Arrow extension}
	\label{arrow}
\end{figure}
If we identify each arc with the corresponding string module, then $[v]=X_{M(v)}$ where we use the same letter for the arc and string. So we have the following:
\begin{theorem}
	If two strings $v$ and $w$ have only one arrow extension, then we have a multiplication formula
	\begin{equation*}
		X_{M(v)}X_{M(w)}=q^{\frac{\Lambda(v,w)-\Lambda(u_1,u_2)}{2}+1}X_{M(u_1)}X_{M(u_2)}+q^{\frac{\Lambda(v,w)-\Lambda(u_3,u_4)}{2}-1}X_{M(u_3)}X_{M(u_4)}
	\end{equation*}
\end{theorem}
\subsection{Overlap extension}
Let $v$ and $w$ have an overlap extension. Recall that in this case, $v=v_Lbma^{-1}v_R$ and $w=w_Ld^{-1}mcw_R$. Figure \ref{overlap} shows the overlap extension on the surface. From the diagram, we can see $u_1=v_Lbmcw_R$ and $u_2=w_Ld^{-1}ma^{-1}v_R$.
\begin{equation*}
	u_3=\left\{
	\begin{aligned}
		v_Lf^{-1}w_L \ \ \ & if \ b\neq \emptyset,d\neq \emptyset \\
		(w_L)_h \ \ \ & if \ b=\emptyset,d\neq \emptyset \\
		(v_L)_c \ \ \ & if \ b\neq \emptyset,d=\emptyset
	\end{aligned}
	\right.
\end{equation*}
\begin{equation*}
	u_4=\left\{
	\begin{aligned}
		v_Re^{-1}w_R \ \ \ & if \ a\neq \emptyset,c\neq \emptyset \\
		_h(w_R) \ \ \ & if \ a=\emptyset,c\neq \emptyset \\
		_c(v_R) \ \ \ & if \ a\neq \emptyset,c=\emptyset
	\end{aligned}
	\right.
\end{equation*}
\begin{figure}[h]
	\centering
	\includegraphics[scale=0.15]{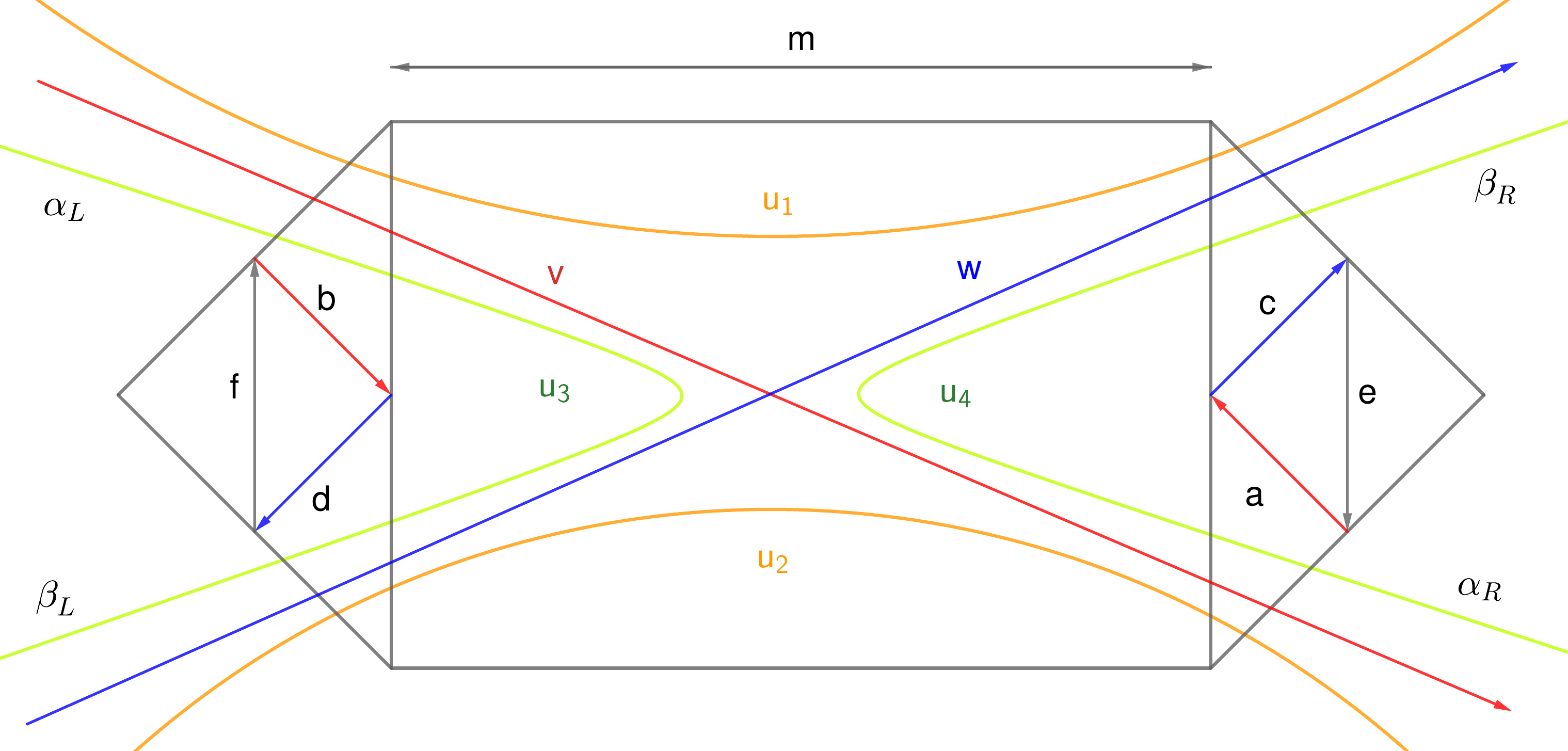}
	\caption{Overlap extension}
	\label{overlap}
\end{figure}
Using the skein relation, we have
\begin{lemma}
	$vw=q^{\frac{\Lambda(v,w)-\Lambda(u_1,u_2)}{2}+1}[u_1][u_2]+q^{\frac{\Lambda(v,w)-\Lambda(u_3,u_4)}{2}-1}[u_3][u_4]$
\end{lemma}
Then if we identify each arc with the corresponding string module, the following is true: 
\begin{theorem}
	If two strings $v$ and $w$ have only one overlap extension, then we have a multiplication formula 
	\begin{equation*}
		X_{M(v)}X_{M(w)}=q^{\frac{\Lambda(v,w)-\Lambda(u_1,u_2)}{2}+1}X_{M(u_1)}X_{M(u_2)}+q^{\frac{\Lambda(v,w)-\Lambda(u_3,u_4)}{2}-1}X_{M(u_3)}X_{M(u_4)}
	\end{equation*}
\end{theorem}

\section{The Kronecker Case}
For the case when the quiver is $1\rightrightarrows 2$, \cite{CANAKCI2020105132} and \cite{Huang2021AnEF} both gave the expansion formula of cluster variables in this quantum cluster algebra. In this section, we show that they are actually the same.
\subsection{Results in \cite{CANAKCI2020105132}}
First, observe that the snake graph corresponding to arcs of this type has alternating face weights (diagonal labels) 1 and 2.
\begin{definition}
	The snake graph $\mathcal{G}_n$ is the following straight snake graph consisting of $2s+1$ tiles with alternating face weights $1$ and $2$, with exactly $s+1$ tiles of weight~$1$ and exactly $s$ tiles of weight~$2$. The snake graph $\mathcal{H}_s$ is obtained from $\mathcal{G}_s$ by removing the last tile (with weight 1).
\end{definition}
\noindent Denote the $2s+1$ tiles of the snake graph $\mathcal{G}_s$ by
\begin{equation*}
	G_{-s},G_{-(s-1)},\dots,G_{s-1},G_s
\end{equation*}
Similarly, for $\mathcal{H}_s$
\begin{equation*}
	H_{-s},H_{-(s-1)},\dots,H_{s-2},H_{s-1}
\end{equation*}
\begin{definition}
	Let $s\geq 0$, define a function on the tiles of the snake graph:
	\begin{enumerate}
		\item $\alpha(G_i)=i$ if $G_i$ has weight 1 and $\alpha(G_i)=-i$ if $G_i$ has weight 2.
		\item $\alpha(H_i)=i+1$ if $H_i$ has weight 1 and $\alpha(H_i)=-i$ if $H_i$ has weight 2.
	\end{enumerate}
\end{definition}
Then for snake graph $\mathcal{G}$ which may be either $\mathcal{G}_s$ or $\mathcal{H}_s$, let $P_-(\mathcal{G})$ be the minimal perfect matching. For every perfect matching $P$, the symmetric difference encloses a set of tiles of $\mathcal{G}$. Let $Twist(P)$ be the set of tiles enclosed. Define a map
\begin{equation*}
	\alpha: \mathcal{P}(\mathcal{G})\rightarrow \mathbb{Z}, P\mapsto \sum_{G\in Twist(P)} \alpha(G)
\end{equation*}
\begin{definition}[The expansion formula]
	For $s\geq 0$, let 
	\begin{equation*}
		r_s=\sum_{P\in \mathcal{P}(\mathcal{G}_s)}q^{\frac{\alpha(P)}{2}}X(P)
	\end{equation*}
\end{definition}
\begin{remark}
	In \cite{CANAKCI2020105132}, they also gave the expansion formula for $\mathcal{H}_s$ 
	\begin{equation*}
		\sum_{P\in \mathcal{P}(\mathcal{H}_s)}q^{\frac{\alpha(P)}{2}}X(P)
	\end{equation*}
	These elements play a key role even though they are not cluster variables.
\end{remark}
\subsection{Equality}
We want to show that the expansion formula above equals the expansion formula in \cite{Huang2021AnEF}, which is 
\begin{equation*}
	\sum_{P\in \mathcal{P}(\mathcal{G}_s)}q^{\frac{v(P)}{2}}X(P)
\end{equation*}
Recall that for $P\in \mathcal{P}(\mathcal{G}_s)$, the symmetric difference actually gives the exponent of $X(P)$, which is the dimension vector $e$ of canonical submodule $M(P)$:
\begin{equation*}
	X(P)=X^{g+B_\Gamma e}
\end{equation*}
where $g$ is the exponent of $X(P_-)$. Therefore, if we can show the following equality is true, then the two expansion formulas are equal.
\begin{equation*}
	\sum_{P:dimP=(u,w)}q^{\frac{\alpha(P)}{2}}=\sum_{P:dimP=(u,w)}q^{\frac{v(P)}{2}} \tag{*}
\end{equation*}
Actually, $u$ and $w$ count the number of tiles with diagonal 1 and 2 in $Twist(P)$. We prove $(*)$ by induction on the number of tiles. The equality is obviously true for $\mathcal{G}_1$ and $\mathcal{H}_1$.
\subsection{In $\mathcal{G}_s$}
Let $P$ be a perfect matching of $\mathcal{G}_s$. If $Twist(P)$ does not contain the last tile, then it can be viewed as a perfect matching of $\mathcal{H}_s$; if $Twist(P)$ contains the last tile, then it can be viewed as a perfect matching of $\mathcal{G}_{s-1}$. The snake graphs and the values of $\alpha$ are shown in Figure \ref{gs}. They are the snake graph of $\mathcal{G}_s,\mathcal{H}_s,\mathcal{G}_{s-1}$ in order, the edges of the minimal perfect matching are shown in blue and the values of $\alpha$ are on the edges.
\begin{figure}[htbp]
	\centering
	\includegraphics[scale=0.1]{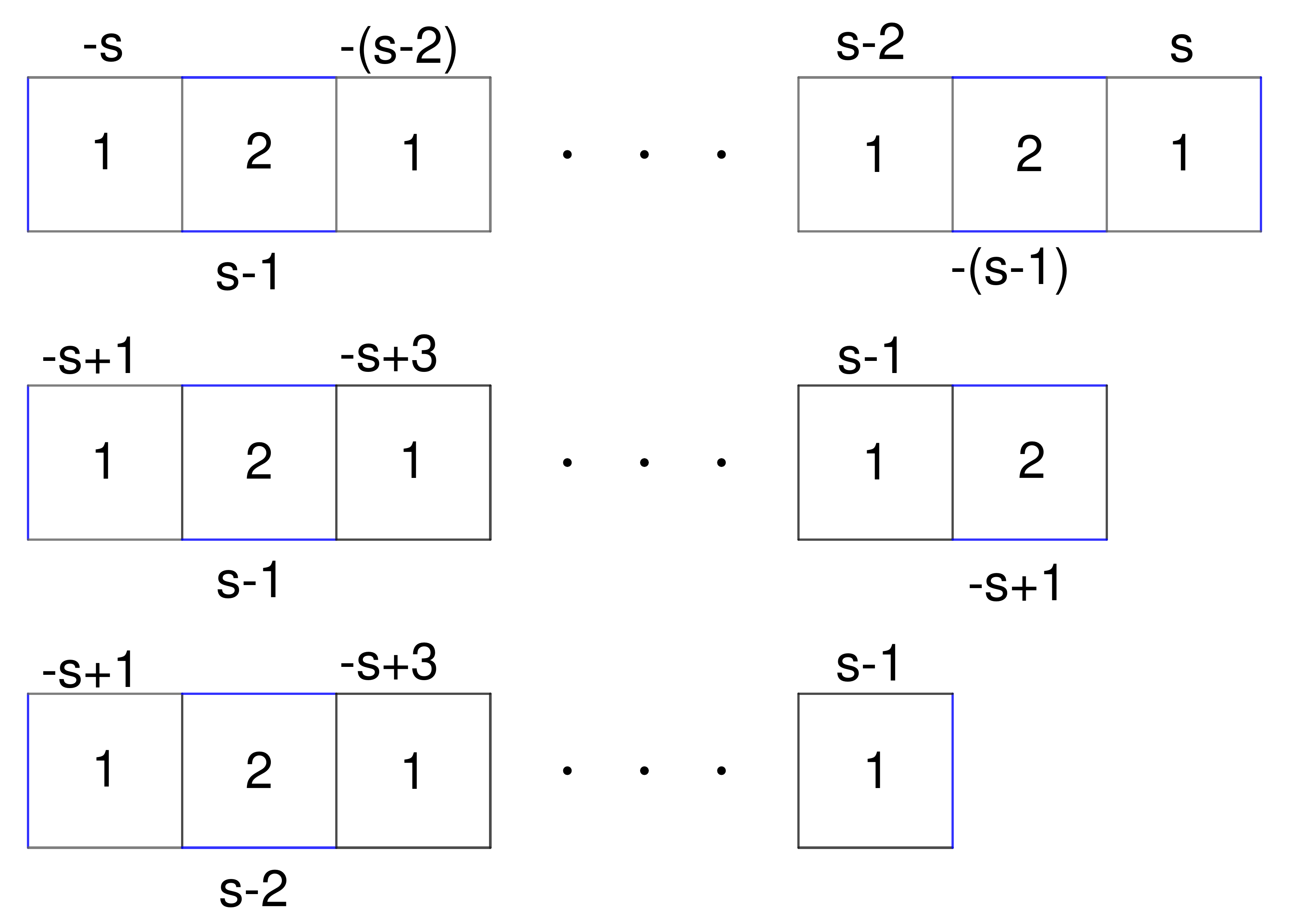}
	\caption{Snake graph and values of $\alpha$}
	\label{gs}
\end{figure}
\begin{lemma}
	Let $P$ be a perfect matching of $\mathcal{G}_s$ with dimension vector $(u,w)$ such that $Twist(P)$ does not contain the last tile. Then $P'$ is a perfect matching of $\mathcal{H}_s$ corresponding to $P$. In addition, $\alpha(P)=\alpha(P')-u$ and $v(P)=v(P')-u$.
\end{lemma}
\begin{proof}
	The correspondence is obvious by definition. We know $P$ and $P'$ have the same edge except for the last tile labeled 1. So the dimension vector of $P'$ is $(u,w)$. The twist tiles of $P$ and $P'$ are the same. Also, there are $u$ tiles labeled 1 and $w$ tiles labeled 2. The values of tiles labeled 1 differ by $-1$ and the values of tiles labeled 2 are the same. Therefore $\alpha(P)=\alpha(P')-u$.
	\par For the latter, let $P_-$ and $P'_-$ be the minimal perfect matchings of $\mathcal{G}_s$ and $\mathcal{H}_s$. Then $v(P)-v(P_-)$ and $v(P')-v(P'_-)$ have the same procedure: twist the tile labeled 1 $u$ times and labeled 2 $w$ times. When twisting at tiles labeled 1, they differ by $-1$ and when twisting at tiles labeled 2, they are the same. So $v(P)=v(P')-u$.
\end{proof}
\begin{lemma}
	Let $P$ be a perfect matching of $\mathcal{G}_s$ with dimension vector $(u,w)$ such that $Twist(P)$ contains the last tile. Then $P''$ is a perfect matching of $\mathcal{G}_{s-1}$ corresponding to $P$. In addition, $\alpha(P)=\alpha(P'')-u+w+1$ and $v(P)=v(P'')-u+w+1$.
\end{lemma}
\begin{proof}
	The correspondence is obvious by definition. We know $P$ and $P''$ have the same edges except  for the last two tiles. So the dimension vector of $P''$ is $(u-1,w-1)$. The twist tiles of $P$ and $P''$ differ by the last two tiles and $\alpha$ maps these tiles to $-(s-1)+s$. Also there are $u-1$ tiles labeled 1 and $w-1$ tiles labeled 2. The values of tiles labeled 1 differ by $-1$ and the values of tiles labeled 2 differ by 1. Therefore $\alpha(P)=\alpha(P'')+1-(u-1)+w-1=\alpha(P'')-u+w+1$.
	\par For the latter, let $P_-$ and $P''_-$ be the minimal perfect matchings of $\mathcal{G}_s$ and $\mathcal{G}_{s-1}$. Then let $Q$ be the perfect matching obtained from $P_-$ by twisting at the last two tiles. By computation, $v(Q)=1$. Then $v(P)-v(Q)$ and $v(P'')-v(P''_-)$ have the same procedure except the last two tiles: twist the tile labeled 1 $u-1$ times and labeled 2 $w-1$ times. When twisting at tiles labeled 1, they differ by $-1$ and when twisting at tiles labeled 2, they differ by $1$. 
	So $v(P)-1=v(P'')-(u-1)+w-1$.
\end{proof}
$\textbf{Proof of (*)}$: Induction on the number of tiles. Assume that $\mathcal{H}_{s}$ and $\mathcal{G}_{s-1}$ satisfy the above.
\begin{align*}
	\sum_{P:dimP=(u,w)}q^{\frac{\alpha(P)}{2}}&=\sum_{1\notin Twist(P)}q^{\frac{\alpha(P)}{2}}+\sum_{1\in Twist(P)}q^{\frac{\alpha(P)}{2}} \\
	&=\sum_{P'}q^{\frac{\alpha(P')-u}{2}}+\sum_{P''}q^{\frac{\alpha(P'')-u+w+1}{2}} \\
	&=\sum_{P'}q^{\frac{v(P')-u}{2}}+\sum_{P''}q^{\frac{v(P'')-u+w+1}{2}} \\
	&=\sum_{1\notin Twist(P)}q^{\frac{v(P)}{2}}+\sum_{1\in Twist(P)}q^{\frac{v(P)}{2}} \\
	&=\sum_{P:dimP=(u,w)}q^{\frac{v(P)}{2}}
\end{align*}
\subsection{In $\mathcal{H}_s$}Let $P$ be a perfect matching of $\mathcal{H}_s$. If $Twist(P)$ contains the last tile, then it can be seen as a perfect matching of $\mathcal{G}_{s-1}$; if $Twist(P)$ does not contain the last tile, then it can be seen as a perfect matching of $\mathcal{H}_{s-1}$. The snake graph and the values of $\alpha$ are shown in Figure \ref{hs}. They are the snake graph of $\mathcal{H}_s,\mathcal{G}_{s-1},\mathcal{H}_{s-1}$ in order, the edges of the minimal perfect matching are shown in blue and the values of $\alpha$ are on the edges.
\begin{figure}[htbp]
	\centering
	\includegraphics[scale=0.1]{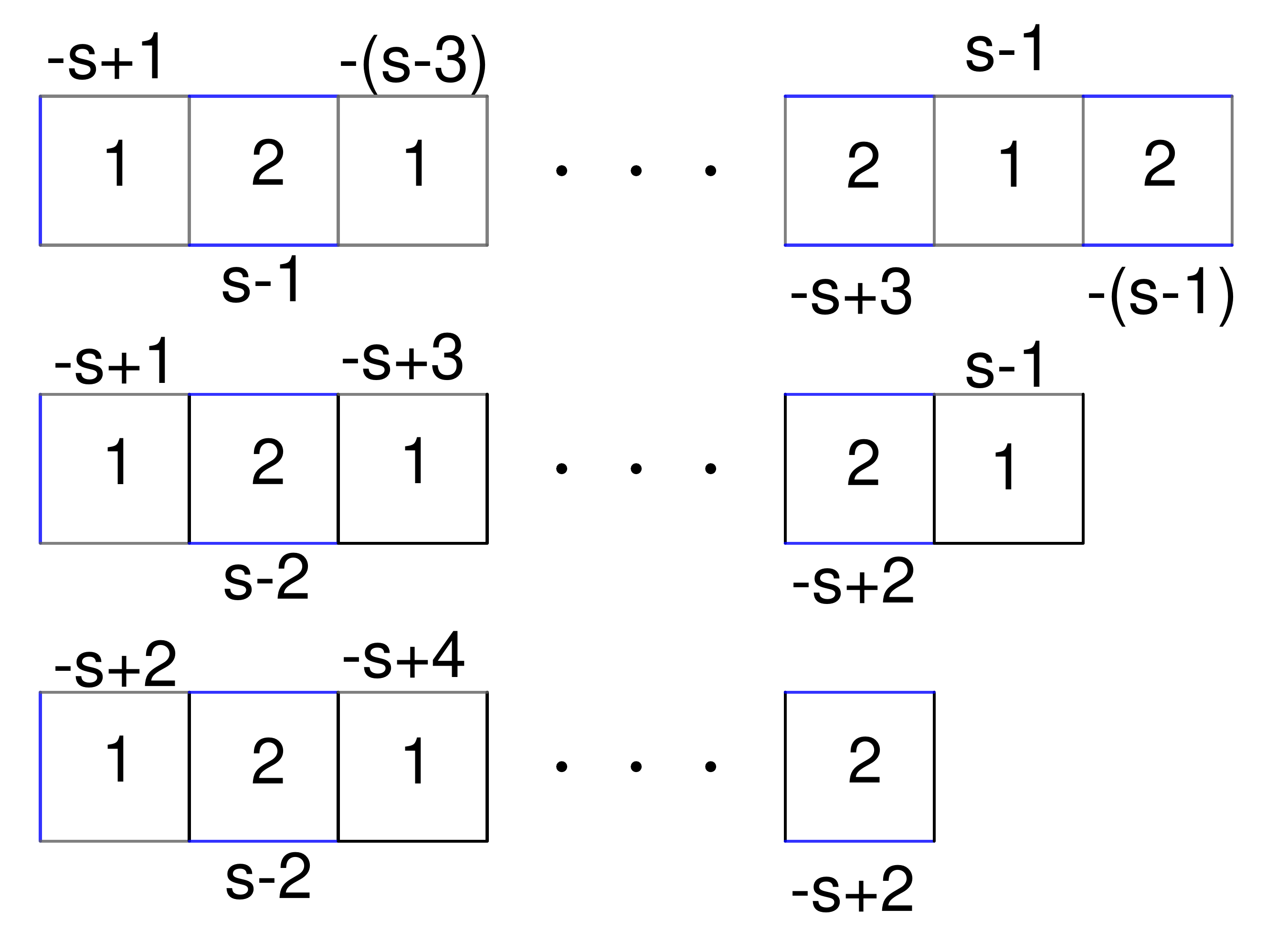}
	\caption{Snake graph and values of $\alpha$}
	\label{hs}
\end{figure}
\begin{lemma}
	Let $P$ be a perfect matching of $\mathcal{H}_s$ with dimension vector $(u,w)$ such that $Twist(P)$ contains the last tile. Then $P'$ is a perfect matching of $\mathcal{G}_{s-1}$ corresponding to $P$. In addition, $\alpha(P)=\alpha(P')-s+w$ and $v(P)=v(P')+s-w$.
\end{lemma}
\begin{proof}
	The correspondence is obvious by definition. We know $P$ and $P'$ have the same edges except for the last tile labeled 1. So the dimension vector of $P'$ is $(u,w-1)$. The twist tiles of $P$ and $P'$ differ by the last tile and $\alpha$ maps this tile to $-s+1$. Also there are $u$ tiles labeled 1 and $w-1$ tiles labeled 2. The values of tiles labeled 1 are the same and the values of tiles labeled 2 differ by $1$. Therefore $\alpha(P)=\alpha(P')-s+1+w-1$.
	\par For the latter, let $P_-$ and $P'_-$ be the minimal perfect matchings of $\mathcal{H}_s$ and $\mathcal{G}_{s-1}$. Then let $Q$ be the perfect matching obtained from $P_-$ by twisting at the last tile. By computation, $v(Q)=s-1$. Then $v(P)-v(Q)$ and $v(P')-v(P'_-)$ have the same procedure: twist the tile labeled 1 $u$ times and labeled 2 $w-1$ times. When twisting at tiles labeled 1, they are the same and when twisting at tiles labeled 2, they differ by $-1$. So $v(P)=v(P')+s-1-(w-1)$.
\end{proof}
\begin{lemma}
	Let $P$ be a perfect matching of $\mathcal{H}_s$ with dimension vector $(u,w)$ such that $Twist(P)$ does not contain the last tile. Then $P''$ is a perfect matching of $\mathcal{H}_{s-1}$ corresponding to $P$. In addition, $\alpha(P)=\alpha(P'')-u+w$ and $v(P)=v(P'')+u-w$.
\end{lemma}
\begin{proof}
	The correspondence is obvious by definition. We know $P$ and $P''$ have the same edges except for the last two tiles. So the dimension vector of $P''$ is $(u,w)$. The twist tiles of $P$ and $P''$ are the same and there are $u$ tiles labeled 1 and $w$ tiles labeled 2. The values of tiles labeled 1 differ by $-1$ and the values of tiles labeled 2 differ by 1. Therefore $\alpha(P)=\alpha(P'')-u+w$.
	\par For the latter, let $P_-$ and $P''_-$ be the minimal perfect matchings of $\mathcal{G}_s$ and $\mathcal{G}_{s-1}$. Then $v(P)-v(P_-)$ and $v(P'')-v(P''_-)$ have the same procedure: twist the tile labeled 1 $u$ times and labeled 2 $w$ times. When twisting at tiles labeled 1, they differ by $1$ and when twisting at tiles labeled 2, they differ by $-1$. 
	So $v(P)-1=v(P'')+u-w$.
\end{proof}
$\textbf{Proof of (*)}$: Induction on the number of tiles. Assume that $\mathcal{G}_{s-1}$ and $\mathcal{H}_{s-1}$ satisfy the above.
\begin{align*}
	\sum_{P:dimP=(u,w)}q^{\frac{\alpha(P)}{2}}&=\sum_{2\in Twist(P)}q^{\frac{\alpha(P)}{2}}+\sum_{2\notin Twist(P)}q^{\frac{\alpha(P)}{2}} \\
	&=\sum_{P'}q^{\frac{\alpha(P')-s+w}{2}}+\sum_{P''}q^{\frac{\alpha(P'')-u+w}{2}} \\
	&=\sum_{P'}q^{\frac{v(P')-s+w}{2}}+\sum_{P''}q^{\frac{v(P'')-u+w}{2}} \\
	&=\sum_{P'}q^{\frac{-v(P')-s+w}{2}}+\sum_{P''}q^{\frac{-v(P'')-u+w}{2}} \\
	&=\sum_{2\in Twist(P)}q^{-\frac{v(P)}{2}}+\sum_{2\notin Twist(P)}q^{-\frac{v(P)}{2}} \\
	&=\sum_{P:dimP=(u,w)}q^{-\frac{v(P)}{2}} \\
	&=\sum_{P:dimP=(u,w)}q^{\frac{v(P)}{2}}
\end{align*}
Here we use the bar-invariant property of the quantum cluster variables.

\end{document}